\documentclass[final]{siamart251216}
\usepackage[english]{babel}
\usepackage[utf8]{inputenc}
\usepackage[T1]{fontenc}
\usepackage{lipsum}
\usepackage{amsfonts, amsmath, amssymb, mathtools}
\usepackage{graphicx}
\usepackage{cleveref}
\usepackage{enumerate}
\usepackage{epstopdf}
\usepackage{algorithmic}
\ifpdf%
  \DeclareGraphicsExtensions{.eps,.pdf,.png,.jpg}
\else
  \DeclareGraphicsExtensions{.eps}
\fi
\usepackage{amsopn}
\DeclareMathOperator{\diag}{diag}
\usepackage{booktabs}

\newcommand{\TheTitle}{Norm-based convergence bounds for nonsymmetric algebraic V-cycle multigrid methods}

\newcommand{\TheShortTitle}{Convergence bounds for V-cycle multigrid methods}

\allowdisplaybreaks
\DeclareMathOperator{\R}{\mathbb{R}}
\DeclareMathOperator{\C}{\mathbb{C}}
\DeclareMathOperator{\N}{\mathbb{N}}

\DeclareMathOperator{\calR}{\mathcal{R}}
\DeclareMathOperator{\calN}{\mathcal{N}}

\DeclareMathOperator{\rank}{rank}

\newcommand{\abs}[1]{\ensuremath{\left\vert#1\right\vert}}
\newcommand{\m}[1]{\begin{bmatrix}#1\end{bmatrix}}
\newcommand{\Cnn}{\C^{n \times n}}
\newcommand{\Cnnc}{\C^{n \times n_c}}
\newcommand{\Cncnc}{\C^{n_c \times n_c}}

\newcommand{\Ekbackp}{E_{k,+}\backslash}
\newcommand{\Ekslashp}{E_{k,+}/}
\newcommand{\Ekbackpnu}{E_{k,+}^{\nu_1}\backslash}
\newcommand{\Ekslashpnu}{E_{k,+}^{\nu_2}/}

\newcommand{\Ekoneslashp}{E_{k+1,+}/}
\newcommand{\Ekonebackpnu}{E_{k+1,+}^{\nu_1}\backslash}
\newcommand{\Ekoneslashpnu}{E_{k+1,+}^{\nu_2}/}

\newcommand{\Eplus}{E_+^{\nu_1, \nu_2}}
\newcommand{\Esouth}{E^{\nu_1, \nu_2}}
\newcommand{\tildeMinvB}{\widetilde{M}^{-1}B}
\newcommand{\hatMinvB}{\widehat{M}^{-1}B}
\newcommand{\lambdaquer}{\overline{\lambda}}

\DeclarePairedDelimiter{\norm}{\lVert}{\rVert}
\def\<{\langle}
\def\>{\rangle}
\def\{{\lbrace}
\def\}{\rbrace}

\newsiamremark{remark}{Remark}
\newsiamremark{example}{Example}
\AddToHook{env/remark/begin}{\crefalias{theorem}{remark}}
\AddToHook{env/example/begin}{\crefalias{theorem}{example}}
\crefname{remark}{Remark}{Remarks}
\crefname{example}{Example}{Examples}

\author{Reinhard Nabben\thanks{Institute of Mathematics, Technische Universität Berlin, (\email{nabben@math.tu-berlin.de}).} 
\and Ludwig Rooch\thanks{Institute of Mathematics, Technische Universität Berlin, (\email{rooch@math.tu-berlin.de}).}}
\title{{\TheTitle}\thanks{Submitted to the editors DATE.}}
\headers{\TheShortTitle}{Reinhard Nabben, Ludwig Rooch}
\ifpdf%
\hypersetup{%
  pdftitle={\TheTitle},
  pdfauthor={Reinhard Nabben, Ludwig Rooch}
}
\fi

\begin{document}
\maketitle

\begin{abstract}
Recently a new approach to analyze and create algebraic multigrid methods (AMG) for nonsymmetric and indefinite matrices was established. Convergence is measured in general norms induced by a certain HPD matrix $B$ and $B$-orthogonal projections built by compatible transfer operators are used. Here we continue our theoretical framework, started in \cite{NabRoo2026a}, for nonsymmetric algebraic multigrid methods using any HPD matrix $B$ to induce a norm. Our framework not only includes all recent results but also provides many new results. 
We consider two, slightly different, multigrid operators. The first one is the natural generalization of the error operator in the HPD case. The second operator is simpler to apply and has been studied before. However, an additional condition for the smoother $M^{-1}A$ is needed, which is in our terminology the $B$-normality. We explain the differences and similarities of both operators in detail and show, why the extra condition is needed. We consider arbitrary interpolation and restriction operators that result in $B$-orthogonal coarse-grid corrections and give sharp estimates for the norm of the error propagation matrices for the two-grid methods. We also show, that the norms are decreasing if we increase the size of the coarse space. Moreover, we are able to extend the landmark $V$-cycle bound by McCormick to the nonsymmetric case.    
\end{abstract}



\section{Introduction}\label{sec:introduction}

Algebraic multigrid (AMG) methods are among the most efficient methods to solve linear algebraic systems $Ax = b$ with nonsingular $A \in \Cnn$. 
They were first introduced by Brandt, McCormick and Ruge \cite{BraMcCRug1985} and Ruge and Stüben \cite{RugStu1987, Stu1983}. In contrast to geometric multigrid methods, AMG methods require no prior information about a grid hierarchy and can be applied to any linear system. Traditionally, AMG is used for symmetric (Hermitian) positive definite (SPD or HPD) linear systems. In this case, $A$ induces the $A$-norm (also called the energy norm) which is used as a convergence measure. 

The convergence in the $A$-norm is well understood (see e.g. \cite{McC1985,TroOosSch2001,Vas2008, MacOls2014,Not2015,XuZik2017}) and many result have been proven over time, including sharp bounds for the $A$-norm of the error propagation matrix \cite{FalVas2004, FalVasZik2005} and optimal transfer operators (interpolation and restriction) minimizing the $A$-norm or the spectral radius of the iteration matrix \cite{FalVasZik2005,BraCaoKahFalHu2018, GarNab2019, XuZik2017, AliBraKahKrzSchSou2025}. 

Different AMG methods have been developed for nonsymmetric problems as well \cite{BreManMcCRugSan2010, Lot2017, ManRugSou2018, ManSou2019, WieTumWalGee2014}, but compared to the HPD case there exist relatively few theoretical results about the convergence behavior and optimality of the transfer operators. One of the problem is that $A$ no longer induces a norm in which the error can be measured and it is unclear which norm to use. Additionally, the coarse-grid correction is no longer orthogonal.

For the nonsymmetric case, different choices for an appropriate norm have been investigated in the past, including the $\sqrt{A^HA}$-norm \cite{BreManMcCRugSan2010, ManOlsSchSou2017,ManSou2019}, the $QA$-norm (where $Q = VU^H$ and $U,V$ are the unitary matrices from the singular value decomposition (SVD) $A = U\Sigma V^H$) \cite{BreManMcCRugSan2010, ManSou2019} and the $M$-norm (where $M$ is an HPD smoother) \cite{Xu2022}. Other approaches use the spectral radius as measure of convergence see e.g. \cite{MenNab2008a, MenNab2008b}. Recent works include \cite{Not2010, Not2020, GarKehNab2020}.

In \cite{SouMan2024} Southworth and Manteuffel introduced so-called \textit{compatible} transfer operators, that means the restriction and prolongation operators are constructed such that the resulting coarse-grid correction has a norm equal to one in some appropriate inner product. This consideration is also observed in \cite{Xu2022}, where an HPD smoother $M$ is used to characterize the convergence in the $M$-norm of nonsymmetric AMG when $A$ is (nonsymmetric) positive definite.

The concept of an orthogonal coarse-grid correction is further developed in \cite{BatNab2025}. For an arbitrary HPD matrix $B \in \mathbb{C}^{n \times n}$, the authors define restriction and prolongation operators such that the coarse-grid correction is $B$-orthogonal. Then, a (sharp) inequality for the $B$-norm of the error propagation matrix of the nonsymmetric pre-smoothing two-grid method is established. 
Even more, optimal restriction and prolongation operators minimizing the $B$-norm of the error propagation matrix are established. Since the HPD matrix $B$ is in general arbitrary, special choices such as $B= A^HA$, $B = AA^H$, $B = M$ are considered in the theory given in \cite{BatNab2025}. 

In \cite{AliBraKahKrzSchSou2025} Ali et al. derived pseudo-optimal transfer operators with respect to the  spectral radius of the error propagation matrix of nonsymmetric AMG. However, norm-based convergence or optimality are not discussed. Additionally, these results do not cover the HPD case.

The results in \cite{AliBraKahKrzSchSou2025} are then generalized in \cite{KrzSouWimAliBraKah2025}. Among other results, it is proven in \cite{KrzSouWimAliBraKah2025} that if $M^{-1}A$ is diagonalizable, the left and right eigenvectors of the generalized eigenvalue problem $Ax = \lambda Mx$ generate optimal transfer operators for a certain $N$-norm of the pre- and post-smoothing nonsymmetric AMG using the same smoother. The HPD matrix $N$ is built with the help of eigenvector matrices. 

Hence $B$-orthogonal projections (for special HPD matrices $B$, as in \cite{AliBraKahKrzSchSou2025, KrzSouWimAliBraKah2025} or general HPD matrices $B$ as in \cite{SouMan2024,BatNab2025}) became a powerful tool to analyze nonsymmetric AMG. 

In \cite{NabRoo2026a} we presented a theoretical framework for nonsymmetric algebraic two-grid methods for arbitrary $B$-inner products and induced $B$-norms which naturally includes the HPD case. Here we continue our theoretical framework. Similarly to \cite{NabRoo2026a} we consider
two, slightly different, two-grid and multigrid error operators that consist of pre- and post-smoothing. The first operator is the natural generalization of the error operator in the HPD case. The second operator is simpler to apply and has been studied before. However, for the second error matrix an additional condition is needed for the smoother $M^{-1}A$. This is in our terminology the $B$-normality, a generalization of normality. We explain the differences and similarities of both operators in detail and show why the extra condition is needed. In \cite{NabRoo2026a} optimal interpolation and restriction operators are established for both methods. These operators are constructed by the eigenvectors of certain eigenvalue problems. Here we consider arbitrary interpolation and restriction operators that lead to $B$-orthogonal coarse-grid corrections. We give sharp estimates for the $B$-norm of the error propagation matrix for both two-grid methods. We also prove that the error norms decrease if we increase the size of the coarse space. Moreover, we extend the land-marked V-cycle bounds by McCormick to the nonsymmetric case.

The structure of the paper is given as follows. In \cref{sec:pre}, we list some preliminary results. In \cref{sec:two-grid conv} we continue the work of \cite{NabRoo2026a} by explaining how $B$-normal matrices can be used in the theory of nonsymmetric two-grid methods. We give new error bounds for such methods in \cref{subsec:conv} and analyze the effect of the number of coarse variables on the norm of the error propagation matrices in \cref{subsec:comp}. Finally, in \cref{sec:MG} we consider nonsymmetric V-cycle multilevel methods where we prove McCormick's V-cycle bound in the nonsymmetric case.

\section{Notation and preliminary results}\label{sec:pre}

We consider the linear algebraic system $Ax = b$ with nonsymmetric and indefinite system matrix $A$. If $A$ is Hermitian positive definite (HPD), usually the $A$-norm is used for any norm based estimate. Since $A$ is now generally nonsymmetric and indefinite, we want to consider a different inner product and norm in which we can symmetrize certain operators and matrices. For this purpose, we introduce the notion of a general $B$-inner product where $B \in \Cnn$ is now always assumed to be some HPD matrix. Here, we follow the notation and definitions from \cite{NabRoo2026a}.

The $B$-inner product and the induced $B$-(vector-)norm are defined as 
\begin{displaymath}
    \< x,y\>_B = y^H B x  \quad \text{and} \quad \lVert x \rVert_B = \sqrt{\<x,x\>_B}.
\end{displaymath}
Based on this, we define the $B$-(matrix-)norm as the operator norm of $\norm{\,\cdot\,}_B$. For a matrix $C \in \Cnn$ the $B$-(matrix-)norm is given by
\begin{displaymath}
    \norm{C}_B = \sup_{x\neq 0} \frac{\norm{Cx}_B}{\norm{x}_B} = \norm{B^\frac12CB^{-\frac12}}_2.
\end{displaymath}
Furthermore, we denote the range or image of the matrix $C$ by $\calR(C)$ and the kernel of $C$ by $\calN(C)$. Another key ingredient for any multigrid method are projections.

\begin{definition}
    A matrix $\Pi \in \Cnn$ is called projection if $\Pi^2 = \Pi$. 
\end{definition}

For us, the class of $B$-orthogonal projections is of particular interest.

\begin{definition} \label{def:B-orth proj}
  Let $\mathcal{U} \subseteq \mathbb{C}^n$. The unique operator $\Pi_{\mathcal{U}, \mathcal{U}^{\bot_B}}$ with $\Pi_{\mathcal{U}, \mathcal{U}^{\bot_B}}^2 = \Pi_{\mathcal{U}, \mathcal{U}^{\bot_B}}$ and
  $\mathcal{R}(\Pi_{\mathcal{U}, \mathcal{U}^{\bot_B}}) =\mathcal{\mathcal{U}}$ and $\mathcal{N}(\Pi_{\mathcal{U}, \mathcal{U}^{\bot_B}}) = \mathcal{\mathcal{U}^{\bot_B}}$ is called the $B$-orthogonal projection onto $\mathcal{U}$ along $\mathcal{U}^{\bot_B}$.
\end{definition}
Here, $\mathcal{U}^{\bot_B}$ denotes the $B$-orthogonal complement of $\mathcal U$. The definition immediately implies that a projection $\Pi$ is $B$-orthogonal if and only if
\begin{displaymath}
    \mathcal{R}(\Pi)^{\bot_B} = \mathcal{N}(\Pi) \quad \text{or equivalently} \quad \mathcal{R}(\Pi) \, \bot_B \, \mathcal{N}(\Pi).
\end{displaymath}

One can also prove the following lemma (see \cite[Lemma~1]{SouMan2024}).

\begin{lemma}\label{lem:B-ortho proj}
    A projection $\Pi$ is $B$-orthogonal if and only if 
    \begin{equation}\label{eq:B-ortho proj}
        \Pi = B^{-1} \Pi^H B.
    \end{equation}
\end{lemma}

The following result (cf. \cite[Lemma~3]{SouMan2024} and \cite[Lemma~3.6]{Vas2008}) gives a useful characterization for the $B$-norm of a projection and its connection to $B$-orthogonality.

\begin{lemma} \label{lem:norm proj}
Let $\Pi \in \mathbb{C}^{n \times n}$ be a projection with $\Pi \neq 0$ and $\Pi \neq I_n$. Then it holds $\norm{\Pi}_B = \norm{I-\Pi}_B$ and
\begin{displaymath}
     \lVert \Pi \rVert^2_B = 1 + \sup_{x\in \mathcal{R}(\Pi)^{\bot_B}} \frac{\lVert \Pi x \rVert^2_B}{\lVert x \rVert^2_B}
\end{displaymath}
with $ \lVert \Pi \rVert^2_B = 1 $ if and only if $ \Pi $ is $B$-orthogonal.
\end{lemma}

Based on this idea of $B$-orthogonality, we introduce the general concept of the adjoint with respect to the $B$-inner product. 

\begin{definition}\label{def:B-adjoint B-normal B-ortho}
    Let $A \in \Cnn$. The $B$-adjoint of $A$ is defined as $A^+ = B^{-1}A^H B$. Then we say that $A$ is $B$-normal if $AA^+ = A^+A$ and $A$ is $B$-orthogonal if $A^+=A$. 
\end{definition}

For the standard inner product, that is $B= I_n$, the $I_n$-adjoint reduces to the Hermitian transpose and hence an $I_n$-orthogonal matrix is Hermitian. If $A$ is HPD and we consider the $A$-inner product, the $A$-adjoint is given by $A$ itself. The definition of $B$-orthogonality in \cref{def:B-adjoint B-normal B-ortho} is consistent with the result of \cref{lem:B-ortho proj} because the right hand side of \cref{eq:B-ortho proj} is exactly the $B$-adjoint of $\Pi$. It is also easy to see that a $B$-orthogonal matrix $A$ is $B$-normal as well and satisfies $A^HB = BA $. 

The close relation of the $B$-adjoint to the Hermitian transpose also shows in its other properties.

\begin{proposition}
    Let $A,C \in \Cnn$, then it holds $(A+C)^+ = A^+ + C^+$, $(AC)^+ = C^+A^+$ and $(A^+)^+ = A$.
\end{proposition}

Similarly to classical normality there exist a lot of equivalent characterizations for $B$-normality. In \cite{ElsIkr1998} and \cite{GroJohSaWol1987} detailed lists with equivalent properties of $I_n$-normality can be found. Most of these properties can be generalized to $B$-normal matrices. We state a few selected characterizations that are relevant for this work.

\begin{definition}
    $A$ is called $B$-unitary if $A^H B A = I_n$, i.e. the columns of $A$ are orthonormal in the $B$-inner product.
\end{definition}

For a $B$-unitary matrix $A$ it holds $A^+A = B^{-1}$ and $A^{-1} = A^H B$ as well as $B = A^{-H}A^{-1}$. The last equation already shows that the structure of the matrix $B$ also depends on the underlying matrix $A$. This behavior is described in the fifth characterization for $B$-normality in the following result which is proven in \cite[Theorem~2.9]{NabRoo2026a}.   

\begin{theorem}\label{thm:B-normal_charact}
    The following statements are equivalent: 
    \begin{enumerate}[(1)]
        \item $A$ is $B$-normal.
        \item \label{thm:enum:Aplus p(A)} $A^+ = p(A)$ for some polynomial $p \in \C[z]$.
        \item \label{thm:enum:B-unitarily diag} $A$ is $B$-unitarily diagonalizable, that is there exists a $B$-unitary matrix $U \in \Cnn$ and a diagonal matrix $D\in \Cnn$ such that $A = U D U^{-1}$. 
        \item \label{thm:enum:eigenvector A Aplus} If $x$ is an eigenvector of $A$ to the eigenvalue $\lambda$, then $x$ is an eigenvector of $A^+$ to the eigenvalue $\lambdaquer$.
        \item \label{thm:enum:diag and decomp B} $A$ is diagonalizable with eigendecomposition $A = W\Lambda W^{-1}$ (without loss of generality we consider the eigenvalues and eigenvectors of $A$ ordered so that equal eigenvalues form a single diagonal block in $\Lambda$). Furthermore, using the eigenvector matrix $W$, the matrix $B^{-1}$ has a decomposition $B^{-1} = WDW^H$ where $D$ is an HPD block diagonal matrix with block sizes corresponding to those in $\Lambda$. There also exists a polynomial $p \in \C[z]$ such that $p(\Lambda) = \Lambda^H$. 
    \end{enumerate}
\end{theorem}

The set $\mathcal{B}$ of all possible matrices $B$ such that $A$ can be $B$-normal can be characterized using equivalence  (\labelcref{thm:enum:diag and decomp B}) and is given by
\begin{align}\label{eq:setB}
    \mathcal{B} = \{ (WDW^H)^{-1} \mid \ &D \text{ is an HPD block diagonal matrix} \\ 
    &\text{with block sizes corresponding to those in } \Lambda \}. \nonumber
\end{align}

Using the characterizations from \cref{thm:B-normal_charact} we can conclude some useful properties of $B$-orthogonal matrices (see \cite{NabRoo2026a} for proofs).

\begin{proposition} \label{prop:eigB-ortho}
    Let $A \in \Cnn $ be $B$-orthogonal. Then all eigenvalues of $A$ are real.
\end{proposition}

\begin{proposition}\label{prop:B-orthogonal_iff_eig_real}
    Let $A \in \Cnn$ be $B$-normal. Then $A$ is $B$-orthogonal if and only if all eigenvalues of $A$ are real.
\end{proposition}

\begin{proposition}\label{prop:B-norm lambda_max}
    Let $A \in \Cnn$. Then it holds
    \begin{displaymath}
        \norm{A}_B^2  = \norm{AA^+}_B = \norm{A^+A}_B = \norm{A^+}_B^2 = \lambda_{max}(A^+A) = \lambda_{max}(AA^+),
    \end{displaymath}
    where $\lambda_{max}$ denotes the largest eigenvalue.
\end{proposition}

\begin{remark}
    Similar to $\lambda_{max}$ we define $\lambda_{min}$ as the smallest eigenvalue of a matrix. Both of these quantities are well-defined if the matrix only has real eigenvalues. But since $A^+A$ is $B$-orthogonal for any matrix $A \in \Cnn$, by \cref{prop:eigB-ortho} its eigenvalues are always real. In addition to $\lambda_{min}$, we define $\lambda_{min}^+$ to be smallest positive eigenvalue of a matrix.
\end{remark}

Another important property of $B$-normal matrices is their relation to the spectral radius $\rho$. 
\begin{proposition}\label{prop:B-normal_spectral_radius}
    Let $A \in \Cnn$ be $B$-normal. Then $\norm{A}_B = \rho(A)$. 
\end{proposition}

\section{Algebraic two-grid convergence theory}\label{sec:two-grid conv}

At first, we consider algebraic two-grid methods for solving the linear system $Ax=b$. Such methods consist of two components, a smoothing step and the coarse-grid correction. Smoothing is usually a basic iterative method such as the Jacobi or Gauss-Seidel method. Coarse-grid correction is a projection that projects the system onto a smaller subspace of dimension $n_c <n$. Here, $n_c$ is the number of coarse variables. To pass from $n$ to $n_c$ variables and back, a restriction matrix $R \in \Cnnc$ and an interpolation matrix $P \in \Cnnc$ are required. The coarse-grid matrix is then given by
\begin{displaymath}
    A_c = R^HAP \in \Cncnc.
\end{displaymath}
We always assume $A_c$ to be nonsingular which means that $R$ and $P$ must have full rank. If $A$ is HPD, these two conditions are equivalent if we choose $R=P$. For general nonsymmetric matrices $A$, it is not sufficient for $R$ and $P$ to have full rank to guarantee that $A_c$ is nonsingular. We simply assume that this is the case. For handling singular coarse-grid systems see e.g. \cite{Not2016, KehNab2016}.

With the coarse-grid matrix $A_c$, we define the coarse-grid correction 
\begin{equation*}
    \Pi_A := \Pi_A(P,R) := P(R^HAP)^{-1}R^HA = PA_c^{-1}R^HA \in \Cnn
\end{equation*}
as a projection onto $\calR(P)$ along $\calN(R^HA)$. To denote the dependency on $R$ and $P$ we write $\Pi_A(P,R)$. For our analysis, the $B$-orthogonality of the coarse-grid correction is a key aspect. In many recent papers \cite{SouMan2024,AliBraKahKrzSchSou2025, KrzSouWimAliBraKah2025, BatNab2025, NabRoo2026a}, different choices of $P$ and $R$ are established such that the coarse gird correction is $B$-orthogonal. The following result summarizes several equivalent properties of a $B$-orthogonal coarse-grid correction (see \cite[Lemma~3.1]{NabRoo2026a}).  

\begin{lemma}\label{lem:equiv_cond Pi B-ortho}
    Let $R,P \in \Cnnc$ such that $R^HAP$ be nonsingular. Then $\Pi_A$ is a projection and the following statements are equivalent:
    \begin{enumerate}[(1)]
        \item $\Pi_A$ is $B$-orthogonal.
        \item $\Pi_A^+ = \Pi_A$.
        \item $\calR(\Pi)^{\bot_B} = \calN(\Pi_A)$ or $\calR(\Pi_A) \, \bot_B \, \calN(\Pi_A)$.
        \item $\norm{\Pi_A}_B = \norm{I-\Pi_A}_B = 1$.
        \item $\calR(BP) = \calR(A^HR)$.
        \item \label{lem:enum:Pi B-ortho equal ranges BatNab} $\calR(P) = \calR(B^{-1}A^HR)$.
        \item $\calN(R^HA) = \calN(P^HB)$. 
    \end{enumerate}
\end{lemma}

In \cite{BatNab2025} condition \eqref{lem:enum:Pi B-ortho equal ranges BatNab} is used to obtain a $B$-orthogonal coarse-grid correction. There, given a restriction $R \in \Cnnc$ of full rank, the corresponding interpolation is defined as $P_* = B^{-1}A^HR$. Conversely, given $P \in \Cnnc$, the corresponding restriction is defined as $R_* = A^{-H}BP$. 
For several specific choices of $B$ such $P_*$ and $R_*$ can be obtained easily. These choices include $B = A$ (HPD), $B = A^HA$, $B = M$ (HPD), and $B = QA$, where $Q$ results from an SVD of $A$, see \cite{BatNab2025}. 
Then, using $P_*$ or $R_*$ we obtain
\begin{equation}\label{eq:Pi_A with Pstar}
    \Pi_A(P_*,R) = \Pi_A(P,R_*) = P_*(P^H_* BP_*)^{-1}P_*^HB = \Pi_B(P_*,P_*).
\end{equation}
Choosing $P_*$ symmetrized the coarse-grid correction in the sense that $\Pi_B(P_*,P_*)$ has the same structure as $\Pi_A(P,P)$ in the HPD case. 
Note that $\calR(R)$ and $\calR(P)$ are independent of the choices of their respective bases. Therefore, operators $(P,R)$ such that $\Pi_A(P,R)$ is $B$-orthogonal are not unique.

Aside from the coarse-grid correction, a two-grid error propagation matrix consists of pre- and post-smoothing steps. In our framework that started in \cite{BatNab2025,NabRoo2026a} we consider the following two error propagation operators
\begin{align}
    \Eplus := \Eplus(P,R) &:= (I-(M^{-1}A)^+)^{\nu_2}(I-\Pi_A(P,R))(I-M^{-1}A)^{\nu_1}, \label{eq:Eplus}
    \\ \Esouth := \Esouth(P,R) &:= (I-M^{-1}A)^{\nu_2}(I-\Pi_A(P,R))(I-M^{-1}A)^{\nu_1}.\label{eq:Esouth}
\end{align}
Again, to denote the dependency of $P$ and $R$ we write $\Eplus(P,R)$ and $\Esouth(P,R)$, and omit $P$ and $R$ if their dependency is clear from the context. The exponents $\nu_1, \nu_2 \in \N_0$ represent the number of pre- and post-smoothing steps, respectively, and $M^{-1} \in \Cnn$ is the smoother which we always assume to be nonsingular. Naturally, the error matrices can be split up to separate the pre- and post-smoothing which yields
\begin{displaymath}
    \Eplus = E_+^{0,\nu_2} \cdot E_+^{\nu_1,0} \quad \text{and} \quad \Esouth = E^{0,\nu_2} \cdot E^{\nu_1,0}.
\end{displaymath}

The two error operators differ in the $B$-adjoint of $M^{-1}A$ in the post-smoothing step of $\Eplus$. This introduces the additional dependency of $B$ into the operator but it allows for stronger convergence results (cf. \cite{NabRoo2026a}). The operator $\Eplus$ is the natural generalization of the error operator in the HPD case, that is
\begin{equation}\label{eq:error_op HPD case}
    E_{\text{HPD}}^{\nu_1, \nu_2}(P,P) := (I-M^{-H}A)^{\nu_2}(I-P(P^HAP)^{-1}P^HA)(I-M^{-1}A)^{\nu_1}
\end{equation}
because for $A=B$ HPD the $A$-adjoint of $M^{-1}A$ is simply $M^{-H}A$. Together with the usual choice $R=P$, the operator $\Eplus$ reduces to $E_{\text{HPD}}^{\nu_1, \nu_2}$. The fact that $\Eplus$ generalizes the HPD error operator is also supported by the properties of $\Eplus$ which again reduce to known properties of $E_{\text{HPD}}^{\nu_1, \nu_2}$ if we consider $A=B$ to be HPD \cite{FalVas2004}.

The  next  Theorem  shall illustrate that  the theoretical analysis  of $E^{\nu,\nu}$ need  some more assumptions than that of $E_+^{\nu,\nu}$.
\begin{theorem} \label{thm:properties Eplus Esouth}
    Let $\Pi_A$ be $B$-orthogonal and $\nu \in \N_0$. Then it holds 
    \begin{enumerate}[1)] 
        \item $(E_+^{\nu,\nu})^+ = E_+^{\nu,\nu}$.
        \item \label{thm:enum:property Eplus only one smoothing adjoint} $(E_+^{\nu,0})^+ = E_+^{0,\nu}$ and $(E_+^{0,\nu})^+ = E_+^{\nu,0}$
        \item $B^{-1} (E_+^{\nu,0})^H = E_+^{0,\nu}B^{-1}$ and $B^{-1}(E_+^{0,\nu})^H = E_+^{\nu,0}B^{-1}$.
        \item \label{thm:enum:property norm Eplus split} $\norm{E_+^{\nu,\nu}}_B = \norm{E_+^{\nu,0}}_B^2 = \norm{E_+^{0,\nu}}_B^2$.
   \end{enumerate}
    If additionally $M^{-1}A$ is $B$-orthogonal, then the above properties also hold for $E^{\nu,\nu}$.
\end{theorem}

The proof of this result can be found in \cite[Theorem~3.2]{NabRoo2026a}. If $M^{-1}A$ is $B$-orthogonal we have $(M^{-1}A)^+ = M^{-1}A$ and the error operators coincide. If $\nu_1 = 1$ and $\nu_2 = 0$ it holds $E_+^{1,0} = E^{1,0}$. This specific operator has been analyzed in \cite{AliBraKahKrzSchSou2025, BatNab2025, Xu2022, ManSou2019}. For general $\nu_1$ and $\nu_2$ the operator $\Eplus$ was recently studied in \cite{NabRoo2026a} while the operator $\Esouth$ was already investigated in \cite{Not2010, Not2016, KrzSouWimAliBraKah2025}. Additionally, for both error operators the optimal compatible transfer operator are known \cite{NabRoo2026a, BatNab2025, KrzSouWimAliBraKah2025, AliBraKahKrzSchSou2025}. Here, compatible means that the resulting coarse-grid correction $\Pi_A$ has a norm equal to one and optimal means that the resulting norm of the error operator is minimal. 

It is well-known that the smoother or symmetrized smoother are crucial in the analysis of HPD multigrid. The same holds in our framework. Therefore, we introduce the following matrices. 
\begin{align}
    \widetilde{M}^{-1} &= M^{-1}AB^{-1} + B^{-1}A^HM^{-H} - B^{-1}A^HM^{-H}BM^{-1}AB^{-1}, \label{eq:tildeMinv}
    \\ \widehat{M}^{-1} &= M^{-1}AB^{-1} + B^{-1}A^HM^{-H}- M^{-1}AB^{-1}A^HM^{-H}, \label{eq:hatMinv}
\end{align}
which satisfy
\begin{align}
    \widetilde M^{-1}B &= M^{-1}A + (M^{-1}A)^+ - (M^{-1}A)^+ M^{-1}A, \label{eq:tildeMinvB}
    \\ \hatMinvB &=  M^{-1}A + (M^{-1}A)^+ - M^{-1}A (M^{-1}A)^+ \label{eq:hatMinvB}
\end{align}
and thus
\begin{align*}
    I-\tildeMinvB &= (I-M^{-1}A)^+(I-M^{-1}A),
    \\ I- \hatMinvB &= (I-M^{-1}A)(I-M^{-1}A)^+.
\end{align*}
The matrices $\widetilde{M}^{-1}$ and $\widehat{M}^{-1}$ are Hermitian. Their structure might seem complicated but for $B=A$ HPD it reduces to the known symmetrized smoother matrices again. Then 
\begin{align*}
    \widetilde M^{-1} &= M^{-1} + M^{-H} - M^{-H}AM^{-1} = M^{-H}(M + M^H - A)M^{-1}, \\
    \widehat M^{-1} &= M^{-1} + M^{-H} - M^{-H}AM^{-1} = M^{-1}(M + M^H - A)M^{-H}.
\end{align*}
These matrices often appear in the AMG analysis of the HPD case (see \cite{XuZik2017, Vas2008, MacOls2014}). Moreover, they are used to determine optimal restriction and prolongation operators \cite{BraCaoKahFalHu2018, XuZik2017, GarNab2019}.

The matrices $\hatMinvB$ and $\tildeMinvB$ not only have a similar structure but also share certain properties as the following result proven in \cite[Proposition~3.3]{NabRoo2026a} shows.
\begin{proposition} \label{prop:properties tildeMinvB hatMinvB}
    Let $\widetilde{M}^{-1} $ and $\widehat{M}^{-1} $ be defined as in \cref{eq:tildeMinv} and \cref{eq:hatMinv}. Then $\tildeMinvB$ and $\hatMinvB$ are $B$-orthogonal and have real eigenvalues. 
    If $M^{-1}A$ is additionally $B$-normal it holds $\widetilde M^{-1} = \widehat M^{-1}$ and $\tildeMinvB = \hatMinvB$. 
\end{proposition}

\Cref{thm:properties Eplus Esouth} shows that $B$-orthogonality of $M^{-1}A$ is required to obtain similar properties for $\Esouth$ and for $\Eplus$. In principal, if $M^{-1}A$ is not $B$-orthogonal, we do not obtain the same properties for $\Esouth$ and for $\Eplus$. 
However, we can   generalize \Cref{thm:properties Eplus Esouth} in the following way.

\begin{theorem} \label{thm:properties Esouth MinvA B-normal}
    Let $\Pi_A$ be $B$-orthogonal, $M^{-1}A$ be $B$-normal and $\nu \in \N$. Then it holds
    \begin{align*}
        \norm{E^{\nu,\nu}}_B &= \norm{E^{\nu,0}}_B^2 = \norm{E^{0,\nu}}_B^2  = \norm{E_+^{\nu,\nu}}_B = \norm{E_+^{\nu,0}}_B^2 = \norm{E_+^{0,\nu}}_B^2 
        \\ &= \lambda_{max}((I-\hatMinvB)^\nu(I-\Pi_A)).
    \end{align*}
\end{theorem}
\begin{proof}
    Let $\Pi_A$ be $B$-orthogonal and $M^{-1}A$ be $B$-normal. By \cref{prop:properties tildeMinvB hatMinvB} we know that $\tildeMinvB = \hatMinvB$ and the $B$-normality of $M^{-1}A$ ensures that
    \begin{align*}
        ((I-M^{-1}A)^\nu)^+(I-M^{-1}A)^\nu &= (I-(M^{-1}A)^+)^\nu(I-M^{-1}A)^\nu 
        \\ &=\left( (I-(M^{-1}A)^+)(I-M^{-1}A)\right)^\nu 
        \\ &=(I-\tildeMinvB)^\nu = (I-\hatMinvB)^\nu
    \end{align*}
    holds, in particular for $\nu >1$. Furthermore, it is
    \begin{align*}
        \norm{E^{\nu,\nu}}_B^2 &= \lambda_{max}\left((E^{\nu,\nu})^+E^{\nu,\nu}\right)
        \\ &=\lambda_{max}\left(((I-M^{-1}A)^{\nu})^+ (I-\Pi_A)^+ ((I-M^{-1}A)^{\nu})^+ \right.
        \\ &\hspace*{2cm} \left.(I-M^{-1}A)^{\nu} (I-\Pi_A)(I-M^{-1}A)^{\nu} \right)
        \\ &= \lambda_{max} \left( (I-\hatMinvB)^\nu (I-\Pi_A) (I-\tildeMinvB)^\nu (I-\Pi_A) \right)
        \\ &= \lambda_{max} \left( ((I-\hatMinvB)^\nu (I-\Pi_A))^2 \right) 
        \\ &= \lambda_{max}^2 ((I-\hatMinvB)^\nu (I-\Pi_A) ). 
    \end{align*}
    Similarly, it holds
    \begin{align*}
        \norm{E^{\nu,0}}_B^2 &=\lambda_{max}\left(((I-M^{-1}A)^{\nu})^+ (I-\Pi_A)^+ (I-\Pi_A) (I-M^{-1}A)^{\nu} \right)
        \\ &= \lambda_{max} ( (I-\hatMinvB)^\nu (I-\Pi_A))
    \end{align*}
    and
    \begin{align*}
        \norm{E^{0,\nu}}_B^2 &=\lambda_{max}\left( (I-\Pi_A)^+ ((I-M^{-1}A)^{\nu})^+ (I-M^{-1}A)^{\nu} (I-\Pi_A) \right)
        \\ &= \lambda_{max} ( (I-\tildeMinvB)^\nu (I-\Pi_A)) = \lambda_{max} ( (I-\hatMinvB)^\nu (I-\Pi_A)).
    \end{align*}
    Finally, with \cref{thm:properties Eplus Esouth}~(\labelcref{thm:enum:property norm Eplus split}) we obtain
    \begin{align*}
        \norm{E_+^{\nu,\nu}}_B &= \norm{E_+^{0,\nu}}_B^2 = \norm{E_+^{\nu,0}}_B^2 
        \\&=\lambda_{max}\left(((I-M^{-1}A)^{\nu})^+ (I-\Pi_A)^+ (I-\Pi_A) (I-M^{-1}A)^{\nu} \right)
        \\ &= \lambda_{max} ( (I-\hatMinvB)^\nu (I-\Pi_A))
    \end{align*}
    which proves the result.
\end{proof}

The above result generalizes not only \Cref{thm:properties Eplus Esouth} but  also the well-known  result for HPD matrices $A$, see e.g.  \cite{FalVas2004}.

One of the necessary assumptions for convergence of AMG is the smoothing assumption. 
\begin{definition}
    We say that the smoothing assumption is satisfied if 
    \begin{displaymath}
        \norm{I-M^{-1}A}_B < 1.
    \end{displaymath}
\end{definition}
The smoothing assumption can be characterized with equivalent properties of the symmetrized smoother matrices $\hatMinvB$ and $\tildeMinvB$ and even further if we assume that $M^{-1}A$ is $B$-normal. Both of the following results are proven in \cite[Theorem~3.5,~Corollary~3.6]{NabRoo2026a}. 

\begin{theorem}\label{thm:smoothing_assumpt_BatNab}
Let $\widetilde M^{-1}$ and $\widehat M^{-1}$ be defined as in \cref{eq:tildeMinv} and \cref{eq:hatMinv}, respectively. Then the following statements are equivalent:
\begin{enumerate}[(1)]
    \item $\lVert I - M^{-1}A\rVert_B < 1$.
    \item $\widetilde{M}^{-1}$ is HPD.
    \item $\sigma(\widetilde{M}^{-1}B) \subseteq (0,1] $.
    \item $\widehat M^{-1}$ is HPD.
    \item $\sigma(\hatMinvB) \subseteq (0,1]$.
\end{enumerate}
\end{theorem}

\begin{corollary}\label{cor:equiv_smoothing_assump}
    Let $\widetilde M^{-1}$ and $\widehat M^{-1}$ be defined as in \cref{eq:tildeMinv} and \cref{eq:hatMinv}, respectively. Let $M^{-1}A$ be $B$-normal. Then the following statements are equivalent:
    \begin{enumerate}[(1)]
        \item $\lVert I - M^{-1}A\rVert_B < 1$.
        \item $\widetilde{M}^{-1} = \widehat M^{-1}$ is HPD.
        \item $\sigma(\widetilde{M}^{-1}B) = \sigma(\hatMinvB) \subseteq (0,1] $.
        \item $\abs{\lambda - 1}^2 < 1$ for all eigenvalues $\lambda$ of $M^{-1}A$.
        \item $\rho(I-M^{-1}A) < 1$
    \end{enumerate}
\end{corollary} 

\Cref{cor:equiv_smoothing_assump} shows that the eigenvalues of $M^{-1}A$ and $\hatMinvB$ are related to each other if $M^{-1}A$ is $B$-normal.  This connection is formalized in the following result. It is proven in \cite[Lemma~3.10]{NabRoo2026a}. 

\begin{lemma}\label{lem:charact_eigenvals_tildeMB}
    Let $M^{-1}A$ be $B$-normal and $(\lambda,z)$ be an eigenpair of $M^{-1}A$. Then $(\mu,z)$ with 
    \begin{displaymath}
        \mu = \lambda + \overline{\lambda} - \overline{\lambda}\lambda = 1- \abs{\lambda-1}^2
    \end{displaymath}
    is an eigenpair of $\hatMinvB = \tildeMinvB$. Furthermore, $M^{-1}A$ and $\hatMinvB = \tildeMinvB$ are simultaneously $B$-unitarily diagonalizable.
\end{lemma}

In \cite{NabRoo2026a} convergence results for both error propagation matrices $\Eplus$ and $\Esouth$ are proven. There, it is shown that similar to the HPD case, the operator $\Eplus$ does not need any further assumptions for the smoothing operator $M^{-1}A$ to achieve a convergent method. The operator $\Esouth$ however requires the $B$-normality of $M^{-1}A$ to yield a similar result. In \cite{NabRoo2026a} this assumption  was thoroughly discussed and demonstrated that it can be seen as a generalization of the assumptions used in \cite{AliBraKahKrzSchSou2025, BatNab2025, KrzSouWimAliBraKah2025}. Examples for matrices $M^{-1}$ and $A$ that result in a $B$-normal matrix $M^{-1}A$ are also given in \cite{AliBraKahKrzSchSou2025,BatNab2025, KrzSouWimAliBraKah2025}. Note that the additional assumption of $B$-normality of $M^{-1}A$ has no counterpart in the HPD case.

\subsection{Convergence results}\label{subsec:conv}
In \cite{NabRoo2026a}  we have established optimal compatible transfer operators.  Here we extend  this convergence analysis by stating more characterizations of the norm of the error operators. For this we  need some auxiliary results. The first result is already known from \cite[Theorem~2.8]{BatNab2025} or \cite[Theorem~2.3]{GarKehNab2020}.

\begin{lemma}\label{lem:spec}
Let $C \in \mathbb{C}^{n \times n}$ be nonsingular and $\Pi \in \mathbb{C}^{n \times n}$, $\Pi \neq 0$, $\Pi \neq I_n$ be a projection with $\dim(\calN(\Pi)) = n_c <n$. Then it holds 
\begin{displaymath}
    \sigma (C \Pi) = \{ 0\} \cup \{ \lambda_{n_c+1}, \dots , \lambda_n\}
\end{displaymath}
if and only if 
\begin{displaymath}
    \sigma (C \Pi + (I - \Pi)) = \{ 1\} \cup \{ \lambda_{n_c+1}, \dots , \lambda_n\} .
\end{displaymath}
\end{lemma}

Another useful result was already proven in \cite[Corollary~2.9]{BatNab2025}.
\begin{lemma}\label{lem:cor2.9 BatNab}
    Let $C \in \Cnn$ be Hermitian, $\sigma(C) \subseteq (0,1)$ and $\Pi \in \Cnn$ be a Hermitian projection with $\Pi \neq 0$ and $\Pi \neq I_n$.
    Then it holds
    \begin{displaymath}
        \rho(I-C\Pi - (I-\Pi)) = \lambda_{max}(I-C\Pi- (I- \Pi)) = 1- \lambda_{min}^+(C\Pi)
    \end{displaymath}
    with $\lambda_{min}^+(C\Pi) \in (0,1)$. Additionally, it holds $\sigma(I-C\Pi -(I-\Pi)) \subseteq [0,1)$.
\end{lemma}
Here, we generalize this result so that it fits into our framework of arbitrary $B$-inner products. Thus, we only assume that the matrix $C$ is $B$-normal rather than Hermitian.

\begin{proposition}\label{prop:cor2.9 B-normal}
    Let $C \in \Cnn$ be $B$-normal, $\sigma(C) \subseteq (0,1)$, $\Pi \in \Cnn$ be a $B$-orthogonal projection with $\Pi \neq 0$ and $\Pi \neq I_n$. Then it holds
    \begin{displaymath}
        \rho(I-C\Pi - (I-\Pi)) = \lambda_{max}(I-C\Pi- (I- \Pi)) = 1- \lambda_{min}^+(C\Pi)
    \end{displaymath}
    with $\lambda_{min}^+(C\Pi) \in (0,1)$. Additionally, it holds $\sigma(I-C\Pi -(I-\Pi)) \subseteq [0,1)$.
\end{proposition}
\begin{proof}
    Let $C$ be $B$-normal and $\Pi$ be a $B$-orthogonal projection with $\Pi \neq 0$ and $\Pi \neq I_n$. Then $\widetilde C = B^\frac12 C B^{-\frac12}$ is normal with $\sigma(\widetilde C) = \sigma(C) \subseteq (0,1)$. Thus, $\widetilde C$ is Hermitian. Similarly, the matrix $\widetilde \Pi = B^\frac12 \Pi B^{-\frac12}$ is a projection and Hermitian due to the $B$-orthogonality of $\Pi$. It holds $\sigma(\widetilde C \widetilde \Pi) = \sigma(C\Pi)$ and $\sigma(\widetilde \Pi) = \sigma(\Pi)$ and hence
    \begin{equation*}
         \sigma(I-C\Pi - (I-\Pi)) = \sigma(I-\widetilde C \widetilde \Pi - (I-\widetilde \Pi)).
    \end{equation*}
    Since $C$ is Hermitian, we can apply \cref{lem:cor2.9 BatNab} and obtain
    \begin{align*}
        \rho(I-C\Pi - (I-\Pi)) &= \rho( I-\widetilde C \widetilde \Pi - (I-\widetilde \Pi)) = 1 - \lambda_{min}^+(\widetilde C \widetilde \Pi) = 1 -\lambda_{min}^+(C\Pi)
        \\ & = \lambda_{max}(I-\widetilde C \widetilde \Pi - (I-\widetilde \Pi)) = \lambda_{max}(I-C\Pi - (I-\Pi))
    \end{align*}
    with $\lambda_{min}^+(C\Pi) = \lambda_{min}^+(\widetilde C \widetilde \Pi) \in (0,1)$. Furthermore, \cref{lem:cor2.9 BatNab} also implies that $\sigma(I-C\Pi -(I-\Pi)) \subseteq [0,1)$.
\end{proof}

This result will be important later in the characterization of the $B$-norm of the error propagation matrices. Another useful result is the reduction of an arbitrary number of smoothing steps into one smoothing step.

\begin{lemma}\label{lem:smoothingstepreduction}
    Let $M^{-1}A$ be $B$-normal, $I-M^{-1}A$ be nonsingular and the smoothing assumption be satisfied. Then for any $\nu \in \N$ the matrix $I-\hatMinvB$ is nonsingular and we have
    \begin{equation*}
        (I-M^{-1}A)^{\nu}(I-(M^{-1}A)^+)^{\nu} = (I-\hatMinvB)^{\nu} = I- \widehat X^{-1}B,
    \end{equation*}
    where $\widehat X^{-1}B$ is $B$-orthogonal, $\sigma(\widehat X^{-1}B) \subseteq (0,1)$ and $\widehat X^{-1}$ HPD. Moreover, there exists a nonsingular matrix $X^{-1}$ such that
    \begin{equation*}
        (I-M^{-1}A)^{\nu} = I-X^{-1}A,
    \end{equation*}
    where $X^{-1}A$ is $B$-normal and 
    \begin{equation} \label{eq:Xhat}
        \widehat X^{-1} = X^{-1}AB^{-1} - B^{-1} (X^{-1}A)^H - X^{-1}AB^{-1}(X^{-1}A)^H.
    \end{equation}
\end{lemma}
\begin{proof}
    Let $M^{-1}A$ be $B$-normal, $I-M^{-1}A$ be nonsingular and the smoothing assumption be satisfied. Then $M^{-1}A$ is $B$-unitarily diagonalizable, that is $M^{-1}A = V\Lambda V^{-1}$ with $V^HBV = I$ and diagonal matrix $\Lambda$ containing the eigenvalues $\lambda_i$. By \cref{lem:charact_eigenvals_tildeMB} we know that $\hatMinvB$ is $B$-orthogonal with $\hatMinvB = V(\Lambda + \Lambda^H - \Lambda\Lambda^H)V^{-1}$ and 
    \begin{equation*}
        I-\hatMinvB = V(I-\Lambda)(I-\Lambda)^HV^{-1} =: VDV^{-1}
    \end{equation*}
    where $D$ is a diagonal matrix with entries $\abs{1-\lambda_i}^2 \in (0,1)$. Therefore, it holds $\sigma(\hatMinvB) \subseteq (0,1)$. Additionally, we have
    \begin{equation*}
        (I-M^{-1}A)^{\nu} = I-X^{-1}A
    \end{equation*}
    for $X^{-1}A = p(M^{-1}A)$ with some real polynomial $p \in \R[z]$ of degree $\nu$. From $B^{-1} = VV^H$ we obtain 
    \begin{equation*}
        X^{-1}A = V(I-(I-\Lambda)^{\nu})V^{-1} 
    \end{equation*}
    and thus the eigenvalues of $X^{-1}A$ are given by $\gamma_i = 1- (1-\lambda_i)^{\nu}$ with 
     \begin{displaymath}
         \abs{1-\gamma_i} = \abs{1-\lambda_i}^\nu < 1.
     \end{displaymath}
     The invertibility of $I-M^{-1}A$ implies that $I-X^{-1}A$ is also nonsingular. Since $p$ has real coefficients we get $(X^{-1}A)^+ = p((M^{-1}A)^+) $. Then we have
    \begin{align*}
        (I-\hatMinvB)^{\nu} &= (I-M^{-1}A)^{\nu}(I-(M^{-1}A)^+)^{\nu} 
        \\&= (I-X^{-1}A)(I-(X^{-1}A)^+) = I-\widehat X^{-1}B 
    \end{align*}
    for 
    \begin{equation*}
        \widehat X^{-1}B = X^{-1}A + (X^{-1}A)^+ - X^{-1}A(X^{-1}A)^+ 
    \end{equation*}
    and 
    \begin{equation*}
        \widehat X^{-1} = p(M^{-1}A)B^{-1} - B^{-1} p(M^{-1}A)^H - p(M^{-1}A)B^{-1}p(M^{-1}A)^H.
    \end{equation*}
    Since $X^{-1}A$ is $B$-unitarily diagonalizable, it is $B$-normal (see \cref{thm:B-normal_charact}~(\labelcref{thm:enum:B-unitarily diag})). The smoothing assumption then implies 
    \begin{displaymath}
        \norm{I-X^{-1}A}_B = \norm{(I-M^{-1}A)^{\nu}}_B \leq \norm{I-M^{-1}A}^{\nu}_B <1
    \end{displaymath}
    and it follows that $\widehat X^{-1}$ is HPD and $\sigma(\widehat X^{-1}B) \subseteq (0,1]$. Since $I-X^{-1}A$ is nonsingular, so is $ I-\widehat X^{-1}B$ and thus $\sigma(\widehat X^{-1}B) \subseteq (0,1)$. $\widehat X^{-1}B$ is $B$-orthogonal by its construction as it has the same structure as $\hatMinvB$. 
\end{proof}

With the above results and \cref{thm:properties Esouth MinvA B-normal} we can now prove another characterization for the $B$-norm of the error operator $\Esouth$.

\begin{theorem}\label{thm:B-norm Esouth 1-lambdamin} 
    Let the smoothing assumption be satisfied, $M^{-1}A$ be $B$-normal and $I-M^{-1}A$ be nonsingular. Further, let $\Pi_A$ be a $B$-orthogonal projection with $\Pi_A\neq 0$ and $\Pi_A \neq I_n$ and $\nu \in \N$. Then it holds 
    \begin{displaymath}
        \norm{E^{\nu,\nu}}_B = \norm{E^{\nu,0}}_B^2 = \norm{E^{0,\nu}}_B^2 = 1 - \lambda_{min}^+(\widehat{X}^{-1}B(I - \Pi_A))
    \end{displaymath}
    where $\widehat{X}^{-1}B = I-(I-\hatMinvB)^\nu$ and $\lambda_{min}^+(\widehat{X}^{-1}B(I - \Pi_A)) \in (0,1)$. Additionally, $\widehat X^{-1}B$ is $B$-orthogonal, $\widehat X^{-1}$ is HPD and $\sigma(\widehat{X}^{-1}B(I - \Pi_A)) \subseteq [0,1)$.
\end{theorem}
\begin{proof}
    Let the smoothing assumption be satisfied, $M^{-1}A$ be $B$-normal, $I-M^{-1}A$ be nonsingular, $\Pi_A$ be a $B$-orthogonal projection and $\nu \in \N$. \Cref{lem:smoothingstepreduction} implies that $I-\hatMinvB$ is nonsingular and there exists a $B$-orthogonal matrix $\widehat X^{-1}B$ with $\widehat X^{-1}B = I-(I-\hatMinvB)^\nu$ and $\sigma(\widehat X^{-1}B) \subseteq (0,1)$. Additionally, $\widehat X^{-1}$ is HPD. From the smoothing assumption and \cref{cor:equiv_smoothing_assump} we obtain $\sigma(\hatMinvB) \subseteq (0,1)$. Finally, \cref{thm:properties Esouth MinvA B-normal} together with \cref{prop:cor2.9 B-normal} imply that
    \begin{align*}
        \norm{E^{\nu,\nu}}_B &= \norm{E^{0,\nu}}_B^2 = \norm{E^{\nu,0}}_B^2 = \lambda_{max}((I-\hatMinvB)^\nu (I-\Pi_A)) 
        \\ &= \lambda_{max}((I-\widehat X^{-1}B)(I-\Pi_A)) = \lambda_{max}(I-\widehat X^{-1}B(I-\Pi_A) -\Pi_A) 
        \\ &= 1- \lambda_{min}^+(\widehat X^{-1}B(I-\Pi_A))
    \end{align*}
    with $\lambda_{min}^+(\widehat X^{-1}B(I-\Pi_A)) \in (0,1)$. Notice that we used \cref{prop:cor2.9 B-normal} with $\Pi = I-\Pi_A$. By \cref{prop:cor2.9 B-normal} we know that
    \begin{equation*}
        \sigma((I-\widehat X^{-1}B)(I-\Pi_A)) = \sigma(I-\widehat X^{-1}B(I-\Pi_A) -\Pi_A) \subseteq [0,1)
    \end{equation*}
    and with \cref{lem:spec} it follows that $\sigma(X^{-1}B(I-\Pi_A)) \subseteq [0,1)$.
\end{proof}

\begin{remark}
    For $\nu = 1$, it is $\widehat X^{-1}B = \hatMinvB$. The assumption that $\Pi_A \neq 0$ and $\Pi_A \neq I_n$ is only a technical requirement to apply \cref{prop:cor2.9 B-normal}. In practice, any multigrid method will satisfy this as $0 < \dim(\calR(\Pi_A)) = n_c < n$ is vacuously true. Since the eigenvalues of $\widehat X^{-1}B(I-\Pi_A)$ are real, we can assume that they are ordered so that
    \begin{displaymath}
        0 = \lambda_1 \leq \dots \leq 0 = \lambda_{n_c} \leq \lambda_{n_c+1} \leq \dots \leq \lambda_n < 1
    \end{displaymath}
    where $\lambda_1, \dots, \lambda_{n_c}$ are the $n_c$ zero eigenvalues and $\lambda_{n_c+1}, \dots, \lambda_n$ are the nonzero eigenvalues. Then we obtain
    \begin{displaymath}
        \norm{E^{\nu,\nu}}_B = 1 - \lambda_{min}^+(\widehat{X}^{-1}B(I - \Pi_A)) = 1- \lambda_{n_c+1}(\widehat{X}^{-1}B(I - \Pi_A)).
    \end{displaymath}
\end{remark}

\begin{remark}\label{rem:B-normal and B-ortho}
    It is important to notice that the $B$-normality of $M^{-1}A$ and the $B$-orthogonality of $\Pi_A$ restrict the choice of possible $B$-norms that can be considered. By \cref{thm:B-normal_charact}~(\labelcref{thm:enum:diag and decomp B}) we know that each possible matrix $B$ can be written as $B = (VD_1V^H)^{-1}$ and $B = (WD_2W^H)^{-1}$. Here, $V$ are the eigenvectors of $M^{-1}A$ and $D_1$ is some HPD block diagonal matrix with block sizes corresponding to those of the eigenvalue matrix $\Lambda$ of $M^{-1}A$ (we assume without loss of generality that the eigenvalues are ordered so that the eigenvalues are clustered into blocks). $W$ are the eigenvectors of $\Pi_A$ and $D_2$ is some HPD block diagonal matrix with two blocks of size $n_c$ and $n-n_c$. 
\end{remark}

We can easily achieve a similar result for the error propagation matrix $E_+^{1,1}$ without assuming that $M^{-1}A$ is $B$-normal.

\begin{theorem}\label{thm:B-norm Eplus 1-lambdamin}
    Let the smoothing assumption be satisfied, $\Pi_A$ be a $B$-orthogonal projection with $\Pi_A\neq 0$ and $\Pi_A \neq I_n$ and $I-\hatMinvB$ be nonsingular. Then it holds 
    \begin{displaymath}
        \norm{E_+^{1,1}}_B = \norm{E_+^{1,0}}_B^2 =\norm{E_+^{0,1}}_B^2 = 1- \lambda_{min}^+(\hatMinvB(I-\Pi_A))
    \end{displaymath}
    with $\lambda_{min}^+(\hatMinvB(I - \Pi_A)) \in (0,1)$. Additionally, it holds $\sigma(\hatMinvB(I - \Pi_A)) \subseteq [0,1)$.
\end{theorem}
\begin{proof}
    Let the smoothing assumption be satisfied, $\Pi_A$ be a $B$-orthogonal projection with $\Pi_A\neq 0$ and $\Pi_A \neq I_n$ and $I-\hatMinvB$ be nonsingular. Then $\sigma(\hatMinvB) \subseteq (0,1)$ and with \cref{prop:cor2.9 B-normal} and \cref{thm:properties Eplus Esouth} we get 
    \Cref{lem:smoothingstepreduction} implies that $I-\hatMinvB$ is nonsingular and there exists a $B$-orthogonal matrix $\widehat X^{-1}B$ with $\widehat X^{-1}B = I-(I-\hatMinvB)^\nu$ and $\sigma(\widehat X^{-1}B) \subseteq (0,1)$. From the smoothing assumption and \cref{cor:equiv_smoothing_assump} we obtain $\sigma(\hatMinvB) \subseteq (0,1)$. Finally, \cref{thm:properties Esouth MinvA B-normal} together with \cref{prop:cor2.9 B-normal} imply that
    \begin{align*}
        \norm{E_+^{1,1}}_B &= \norm{E_+^{1,0}}_B^2 = \norm{E_+^{0,1}}_B^2
        \\&= \lambda_{max}((I-(M^{-1}A)^+)(I-\Pi_A)^+(I-\Pi_A)(I-M^{-1}A))
        \\ &= \lambda_{max}((I-\hatMinvB)(I-\Pi_A)) = \lambda_{max} (I-\hatMinvB(I-\Pi_A)-\Pi_A) 
        \\ &= 1- \lambda_{min}^+(\hatMinvB(I-\Pi_A))
    \end{align*}
    with $\lambda_{min}^+(\hatMinvB(I-\Pi_A)) \in (0,1)$. Notice that we again used \cref{prop:cor2.9 B-normal} with $\Pi = I-\Pi_A$. By \cref{prop:cor2.9 B-normal} we know that
    \begin{equation*}
        \sigma((I-\widehat M^{-1}B)(I-\Pi_A)) = \sigma(I-\widehat M^{-1}B(I-\Pi_A) -\Pi_A) \subseteq [0,1)
    \end{equation*}
    and with \cref{lem:spec} it follows that $\sigma(M^{-1}B(I-\Pi_A)) \subseteq [0,1)$.
\end{proof}

In addition to the above characterizations, the convergence rate of a two-grid method is also described by a certain constant $K$. For two HPD matrices $G$ and $C$ we define

\begin{align}
    K_{G,C}(P) & = \underset{v\in\mathbb{C}^n\setminus\{0\}}{\max}\frac{v^HG(I-P\big(P^HGP)^{-1}P^HG\big)v}{v^HCv} \nonumber 
    \\& = \Bigg(\underset{v\in\mathbb{C}^n\setminus\{0\}}{\max}\frac{\norm{(I-P(P^HGP)^{-1}P^HG)v}_G}{\norm{v}_{C}}\Bigg)^2. \label{eq:def KG}
\end{align}

The constant does not actually depend on the chosen interpolation operator $P$, but rather on its image $\calR(P)$. However, to denote that there still is some dependence on the underlying interpolation, we write $K_{G,C}(P)$. This constant is used quite often in the analysis of multigrid methods for HPD matrices, see e.g. \cite{FalVasZik2005, MacOls2014,XuZik2017}). With this constant $K_{G,C}$ we can prove the following results that again generalize the known results from the HPD case to the nonsymmetric setting.
\begin{theorem} \label{thm:B-norm Esouth 1-K}
    Let the smoothing assumption be satisfied, $M^{-1}A$ be $B$-normal and $I-M^{-1}A$ be nonsingular. Further, let $\Pi_A(P_*,R)$ be a $B$-orthogonal projection with $0 < \dim(\calR(\Pi_A(P_*,R)) = n_c < n$ and $\nu \in \N$. Then it holds
    \begin{equation*}
        \norm{E^{\nu, \nu}}_B = \norm{E^{\nu,0}}_B^2 = \norm{E^{0,\nu}}_B^2 = 1 - \frac{1}{K_{\widehat{X},B}(P_*)}
    \end{equation*}
    where $\widehat X = B(I-(I-\hatMinvB)^\nu)^{-1}$ is HPD and $K_{\widehat{X},B}(P_*) \in (0,1)$.
\end{theorem}
\begin{proof}
    Let the smoothing assumption be satisfied, $M^{-1}A$ be $B$-normal, $I-M^{-1}A$ be nonsingular, $\Pi_A(P_*,R)$ be a $B$-orthogonal projection with $\dim(\calR(\Pi_A)) = n_c$ and $\nu \in \N$. Since $0< n_c <n$, by \cref{thm:B-norm Esouth 1-lambdamin} it holds
    \begin{displaymath}
        \norm{E^{\nu,\nu}}_B = \norm{E^{\nu,0}}_B^2 = \norm{E^{0,\nu}}_B^2 = 1 - \lambda_{min}^+(\widehat{X}^{-1}B(I - \Pi_A))
    \end{displaymath}
    with $\widehat{X}^{-1}B = I-(I-\hatMinvB)^\nu$ and $\lambda_{min}^+(\widehat{X}^{-1}B(I - \Pi_A)) \in (0,1)$. Furthermore, $\widehat X$ is HPD and $\sigma(\widehat X^{-1}B(I-\Pi_A)) \subseteq [0,1)$. Thus we can assume that
    \begin{equation*}
        \sigma(\widehat{X}^{-1}B(I -\Pi_A)) := \{0\} \cup \{ \lambda_{n_c+1}, \ldots, \lambda_n\}
    \end{equation*}
    where $0 < \lambda_{n_c+1} \leq  \dots \leq  \lambda_n<1$ are the nonzero eigenvalues. For convenience, let
    \begin{equation*}
        Q_B := P_*(P_*^HBP_*)^{-1}P_*^H
    \end{equation*}
    where we omit the dependence on $P_*$. Then we have $\Pi_A(P_*,R) = Q_B B$. With \cref{lem:spec} it follows 
    \begin{equation*}
        \sigma(\widehat{X}^{-1}B(I -\Pi_A) + \Pi_A) := \{1\} \cup \{ \lambda_{n_c+1}, \ldots, \lambda_n\}
    \end{equation*}
    and hence the nonzero eigenvalues of $\widehat{X}^{-1}B(I -\Pi_A)$ are the eigenvalues different from $1$ of
    \begin{equation*}
        T:= \widehat{X}^{-1}B(I - \Pi_A) + \Pi_A = \widehat{X}^{-1}B(I - Q_B B) + Q_B B.
    \end{equation*}
    Then $T$ is nonsingular. For $Q_{\widehat {X}} := P_*(P_*^H\widehat{X}P_*)^{-1}P_*^H$ it holds $Q_B\widehat{X}Q_{\widehat {X}} = Q_B $ as well as $Q_BBQ_{\widehat {X}} = Q_{\widehat {X}}$. A straightforward calculation then shows (cf. \cite[Theorem~2.5]{GarKehNab2020}) that
    \begin{equation*}
        T^{-1} = B^{-1}\widehat{X}(I - Q_{\widehat {X}}\widehat{X}) + Q_{\widehat {X}}\widehat{X}. 
    \end{equation*}
    This implies 
    \begin{equation*}
        \sigma(T^{-1}) = \{1\} \cup \left\{ \frac{1}{\lambda_{n_c+1}}, \ldots,  \frac{1}{\lambda_{n}} \right\}
    \end{equation*}
    and since $Q_{\widehat {X}}\widehat{X}$ is a projection we can apply \cref{lem:spec} again to obtain
    \begin{equation} \label{eq:eig2}
        \sigma(B^{-1}\widehat{X}(I - Q_{\widehat {X}}\widehat{X})) = \{0\} \cup \left\{ \frac{1}{\lambda_{n_c+1}}, \ldots,  \frac{1}{\lambda_{n}} \right\}.
    \end{equation}
    With the order of the eigenvalues, we obtain 
    \begin{equation*}
        \frac{1}{\lambda_n} \leq \dots \leq \frac{1}{\lambda_{n_c+1}}
    \end{equation*}
    which implies that 
    \begin{equation*}
        \lambda_{max}(B^{-1}\widehat{X}(I - Q_{\widehat {X}}\widehat{X})) = \frac{1}{\lambda_{n_c+1}} = \frac{1}{\lambda_{min}^+(\widehat X^{-1}B(I-\Pi_A))} \in (0,1)
    \end{equation*}
    The matrices $B$ and $\widehat{X}$ are HPD and $\widehat{X}(I - Q_{\widehat{X}}\widehat{X})$ is Hermitian, we can apply the Courant-Fischer-Theorem and obtain 
    \begin{align*}
        \lambda_{max}(B^{-1}\widehat{X}(I - Q_{\widehat {X}}\widehat{X})) = \sup_{v \neq 0} \frac{v^H\widehat X(I-Q_{\widehat {X}}\widehat{X}) v}{v^HBv} = K_{\widehat X, B}(P_*) \in (0,1).
    \end{align*}
    Finally, we get
    \begin{align*}
        \norm{E^{\nu,\nu}}_B = \norm{E^{\nu,0}}_B^2 = \norm{E^{0,\nu}}_B^2 &= 1 - \lambda_{min}^+(\widehat{X}^{-1}B(I - \Pi_A)) 
        \\ &= 1- \frac{1}{\lambda_{max}(B^{-1}\widehat{X}(I - Q_{\widehat {X}}\widehat{X}))} = 1- \frac{1}{K_{\widehat X, B}(P_*)}.
    \end{align*}
\end{proof}

This result generalizes the known results \cite[Theorem~4.3]{FalVasZik2005}, \cite[Theorem~7]{MacOls2014} and \cite[Theorem~5.4]{XuZik2017} from the HPD case to the nonsymmetric and indefinite case. We can prove a similar result for the operator $\Eplus$.

\begin{theorem}\label{thm:B-norm Eplus 1-K}
    Let the smoothing assumption be satisfied, $I-\hatMinvB$ be nonsingular and $\Pi_A(P_*,R)$ be a $B$-orthogonal projection with $0 < \dim(\calR(\Pi_A(P_*,R)) = n_c < n$. Then it holds
    \begin{equation*}
        \norm{E_+^{1, 1}}_B = \norm{E_+^{1,0}}_B^2 = \norm{E_+^{0,1}}_B^2 = 1 - \frac{1}{K_{\widehat{M},B}(P_*)}
    \end{equation*}
    with $K_{\widehat{M},B}(P_*) \in (0,1)$.
\end{theorem}
\begin{proof}
    The proof is analogous to the proof of \cref{thm:B-norm Esouth 1-K} but we replace $\widehat X$ with $\widehat M$ and use \cref{thm:B-norm Eplus 1-lambdamin} instead of \cref{thm:B-norm Esouth 1-lambdamin}.
\end{proof}

\subsection{Comparison results}\label{subsec:comp}

In \cite{NabRoo2026a} we have established optimal transfer operators that yield the minimal $B$-norms of the error operator $\Eplus$ and $\Esouth$. Aside from choosing the optimal transfer operators, another idea to increase the speed of convergence of AMG is to increase the number of coarse variables $n_c$ to some $\tilde n_c \geq n_c$. The goal is that the increase of information on the coarse-grid gives a better approximation for the solution of the linear system and therefore speeds up the convergence. Clearly, this also means more computational cost and complexity for the error propagation matrix, but the hope is that this approach yields a smaller $B$-norm. 

Unfortunately, this is not the case in general. The following example demonstrates a situation where the $B$-norm grows when we increase the number of coarse variables. The example also works for both operators $\Eplus$ and $\Esouth$.

\begin{example}\label{exam:counterexample_comp}
    We consider $n=3$, $n_c = 1$, $\tilde n_c = 2$ and $\nu_1 = \nu_2 = 1$. Let 
    \begin{equation*}
        A = \diag\left(\frac14,\frac12,\frac34\right), \quad B = I_3 = M^{-1},
    \end{equation*}
    then we have
    \begin{equation*}
        \hatMinvB = \diag\left(\frac7{16}, \frac34, \frac{15}{16}\right), \quad I - \hatMinvB = \diag\left(\frac9{16}, \frac14, \frac{1}{16}\right).
    \end{equation*}
    Notice that $(M^{-1}A)^+ = (M^{-1}A)^H = M^{-1}A$ and thus $M^{-1}A$ is $I_3$-orthogonal. Hence it holds $E_+^{1,1} = E^{1,1}$. We choose 
    \begin{equation*}
        P = R = \m{1\\1\\0}, \quad \text{and} \quad  \tilde P = \m{1&2\\1&1\\0&0}, \ \tilde R = \m{1&0\\1&1\\0&1}.
    \end{equation*}
    Now it holds $\calR(R) \subseteq \calR(\tilde R)$ and $\calR(P) \subseteq \calR(\tilde P)$, but
    \begin{align*}
        \norm{E^{1,1}(P,R)}_B^2 &= \lambda_{max}((I-\hatMinvB)(I-\Pi_A(P,R))^H(I-\hatMinvB)(I-\Pi_A(P,R)))
        \\ &= \lambda_{max}\left(\m{\frac{5}{32} & -\frac{5}{32} & 0 \\[1mm] -\frac{5}{72} & \frac{5}{72} & 0 \\[1mm] 0 &0 & \frac{1}{256} }\right) = \frac{65}{288} \approx 0.2257
    \end{align*}
    and
    \begin{align*} 
        \norm{E^{1,1}(\tilde P, \tilde R)}_B^2 &= \lambda_{max}((I-\hatMinvB)(I-\Pi_A(\tilde P,\tilde R))^H(I-\hatMinvB)(I-\Pi_A(\tilde P,\tilde R)))
        \\ &= \lambda_{max}\left(\m{0 & 0 & 0 \\ 0 & 0 & 0 \\ 0 &0 &\frac{91}{256} }\right) = \frac{91}{256} \approx 0.3555.
    \end{align*}
    This shows
    \begin{equation*}
        \norm{E^{1,1}(P,R)}_B^2 = \norm{E_+^{1,1}(P,R)}_B^2 = \frac{65}{288} \not\geq \norm{E^{1,1}(\tilde P, \tilde R)}_B^2 = \norm{E_+^{1,1}(\tilde P, \tilde R)}_B^2 = \frac{91}{256}.
    \end{equation*}
    Hence in this case we do not obtain a smaller $B$-norm when increasing the number of the coarse variables.
\end{example}

However, it is still possible to guarantee a smaller $B$-norm but with some additional assumptions. Note that in the example above, the projections 
\begin{equation*}
    \Pi_A(P,R) = \m{\frac{2}{3} & -\frac{2}{3} & 0 \\[0.2em] -\frac{1}{3} & \frac{1}{3} & 0 \\[0.2em] 0 & 0 & 1} \quad \text{ and } \quad  \Pi_A(\tilde P, \tilde R)= \m{0 & 0 & 3 \\[0.2em] 0 & 0 &-\frac{3}{2} \\[0.2em] 0 &0 & 1}    
\end{equation*}
are not $I_3$-orthogonal, that is Hermitian. However, the $B$-orthogonality of the projection $\Pi_A$ is exactly the additional assumption we need to guarantee a decrease of the $B$-norm and thereby faster convergence. To prove this we require a well known eigenvalue estimate for Hermitian matrices from \cite{HorJoh2013}.

\begin{lemma}\label{lem:eigenval inequality HornJohnson}
    Let $H_1, H_2 \in \Cnn$ be Hermitian. If $H_2$ is singular, then
    \begin{equation*}
        \lambda_i(H_1 + H_2) \leq \lambda_{i+\rank(H_2)}(H_1)
    \end{equation*}
    for all $i = 1, \dots, n-\rank(H_2)$.
\end{lemma}

We can now prove the desired result for the operator $\Esouth$.

\begin{theorem}\label{thm:comp Esouth}
    Let the smoothing assumption be satisfied, $M^{-1}A$ be $B$-normal and $I-M^{-1}A$ be nonsingular. Let $\nu \in \N$ and $\tilde n_c \in \N$ with $n_c \leq \tilde n_c < n$. Further, let $P,R \in \Cnnc$ and $\tilde P, \tilde R \in \C^{n \times \tilde n_c}$ be such that $R^HAP$ and $\tilde R^HA \tilde P$ are nonsingular and $\Pi_A(P,R)$, $\Pi_A(\tilde P, \tilde R)$ are $B$-orthogonal. If $\calR(R) \subseteq \calR(\tilde R)$ or $\calR(P) \subseteq \calR(\tilde P)$ is fulfilled, then it holds
    \begin{equation*}
        \norm{E^{\nu_1, \nu_2}(\tilde P,\tilde R)}_B \leq \norm{E^{\nu_1, \nu_2}(P,R)}_B
    \end{equation*}
    for any combination $(\nu_1, \nu_2) \in \{(\nu, 0),(0,\nu),(\nu,\nu)\}$.
\end{theorem}
\begin{proof}
    Let the above assumptions be satisfied. Since $\Pi_A(P,R)$ and $\Pi_A(\tilde P, \tilde R)$ are $B$-orthogonal, it follows from \cref{thm:properties Esouth MinvA B-normal} that
    \begin{equation*}
        \norm{E^{\nu,\nu}}_B = \norm{E^{\nu,0}}_B^2 = \norm{E^{0,\nu}}_B^2.
    \end{equation*}
    This is particularly true for $(P,R)$ and $(\tilde P, \tilde R)$. Hence, it is sufficient to show the desired inequality only for the combination $(\nu_1, \nu_2) = (\nu,\nu)$. By \cref{thm:B-norm Esouth 1-lambdamin} we know that it holds
    \begin{equation*}
        \norm{E^{\nu,\nu}(P,R)}_B = 1 - \lambda_{min}^+(\widehat{X}^{-1}B(I - \Pi_A(P,R)))
    \end{equation*}
    and the eigenvalues of $\widehat{X}^{-1}B(I-\Pi_A(P,R))$ are real. 
    Assume that the eigenvalues are ordered increasingly so that $\lambda_i(\widehat{X}^{-1}B(I - \Pi_A(P,R)))$ is the $i$-th smallest eigenvalue of $\widehat{X}^{-1}B(I - \Pi_A(P,R))$. Since $\dim(\calN(\Pi_A(P,R))) = n_c$, the matrix has $n_c$ zero eigenvalues. Using \cref{thm:B-norm Esouth 1-lambdamin} again, it follows
    \begin{equation*}
        \norm{E^{\nu,\nu}(P,R)}_B = 1 - \lambda_{min}^+(\widehat{X}^{-1}B(I - \Pi_A(P,R))) = 1- \lambda_{n_c+1}(\widehat{X}^{-1}B(I - \Pi_A(P,R))).
    \end{equation*}
    Since $\Pi_A(P,R)$ is $B$-orthogonal, by \cref{lem:equiv_cond Pi B-ortho} we have $P = B^{-1}A^HRS$ for some nonsingular matrix $S \in \Cncnc$. Then a straightforward computation shows $\Pi_A(P,R) = \Pi_B(P,P) = P(P^HBP)^{-1}P^HB$ and
    \begin{equation*}
        \norm{E^{\nu,\nu}(P,R)}_B = 1- \lambda_{n_c+1}(\widehat{X}^{-1}B(I - \Pi_A(P,R))) = 1- \lambda_{n_c+1}(\widehat{X}^{-1}B(I - \Pi_B(P,P))).
    \end{equation*}
    Analogously, we can apply the same argument to $\norm{E^{\nu,\nu}(\tilde P, \tilde R)}_B$ and obtain
    \begin{equation*}
        \norm{E^{\nu,\nu}(\tilde P,\tilde R)}_B = 1- \lambda_{\tilde n_c+1}(\widehat{X}^{-1}B(I - \Pi_A(\tilde P,\tilde R))) = 1- \lambda_{\tilde n_c+1}(\widehat{X}^{-1}B(I - \Pi_B(\tilde P,\tilde P))).
    \end{equation*}
    Notice that $\widehat{X}^{-1}B(I - \Pi_B(\tilde P,\tilde P))$ has $\tilde n_c$ zero eigenvalues. We now define the matrix
    \begin{equation*}
        D = \tilde P (\tilde P^H B\tilde P)^{-1}\tilde P^H - P(P^HBP)^{-1}P^H \in \Cnn
    \end{equation*}
    and show that this matrix has rank smaller than $\tilde n_c - n_c < n$. First, notice that
    \begin{equation*}
        DB = \Pi_B(\tilde P, \tilde P) - \Pi_B(P,P) = \Pi_A(\tilde P, \tilde R) - \Pi_A(P,R).
    \end{equation*}
    If $\calR(R) \subseteq \calR(\tilde R)$, the $B$-orthogonality of $\Pi_A(P,R)$ and \cref{lem:equiv_cond Pi B-ortho} imply
    \begin{equation*}
        \calR(P) = \calR(B^{-1}A^HR) \subseteq \calR(B^{-1}A^H \tilde R) = \calR(\tilde P).
    \end{equation*}
    This means that for $\calR(R) \subseteq \calR(\tilde R)$ or $\calR(P) \subseteq \calR(\tilde P)$, it always holds $\calR(P) \subseteq \calR(\tilde P)$ and thus also $\calR(\tilde P)^{\bot_B} \subseteq \calR(P)^{\bot_B}$. Let $\{w_1, \dots, w_{n-\tilde n_c}\}$ be a basis of $\calR(\tilde P)^{\bot_B}$, then we get $\tilde P^HBw_i = 0$ and $P^HBw_i = 0$ for all $i = 1, \dots, n-\tilde n_c$. It follows 
    \begin{equation*}
        DBw_i = \tilde P (\tilde P^H B\tilde P)^{-1}\tilde P^HBw_i - P(P^HBP)^{-1}P^HBw_i = 0 -0 = 0
    \end{equation*}
    for all $i = 1, \dots, n- \tilde n_c$. Now letting $\{v_1, \dots, v_{n_c}\}$ be a basis of $\calR(P)$, it holds $v_i = Pu_i$ and $v_i = \tilde P \tilde u_i$ for some vectors $u_i \in \C^{n_c}$ and $\tilde u_i \in \C^{\tilde n_c}$ for all $i = 1, \dots, n_c$. It follows 
    \begin{align*}
        DBv_i &= \tilde P (\tilde P^H B\tilde P)^{-1}\tilde P^HBv_i - P(P^HBP)^{-1}P^HBv_i
        \\ &= \tilde P (\tilde P^H B\tilde P)^{-1}\tilde P^HB \tilde P \tilde u_i - P(P^HBP)^{-1}P^HB Pu_i = v_i -v_i = 0
    \end{align*}
    for all $i = 1, \dots, n_c$. Furthermore, for all $i = 1, \dots, n_c$ and $j = 1, \dots, n-\tilde n_c$ we have 
    \begin{equation*}
        \<v_i,w_j \>_B = v_i^HBw_i = u_i^HP^HBw_j = 0
    \end{equation*}
    and hence $v_1, \dots, v_{n_c}, w_1, \dots, w_{n-\tilde n_c}$ are linearly independent. Then
    \begin{equation*}
        Bv_1, \dots, Bv_{n_c}, Bw_1, \dots, Bw_{n-\tilde n_c}
    \end{equation*}
    are also linearly independent and are contained in the kernel of $D$. In conclusion, $\dim(\calN(D)) \geq n- \tilde n_c + n_c$ and 
    \begin{equation*}
        \rank(D) = \dim(\calR(D)) = n- \dim(\calN(D)) \leq n - (n- \tilde n_c + n_c) = \tilde n_c - n_c.
    \end{equation*}
    Therefore, $D$ is singular and by construction also Hermitian. We obtain
    \begin{align*}
        \norm{E^{\nu, \nu}(&P,R)}_B = 1 - \lambda_{n_c+1} \big(  \widehat X^{-1}B(I-\Pi_B(P,P)) \big)
        \\ &= 1 - \lambda_{n_c+1} \big(  \widehat X^{-\frac12}B(B^{-1}-P(P^HBP)^{-1}P^H)B \widehat X^{-\frac12}\big)
        \\ &= 1 - \lambda_{n_c+1} \big(  \widehat X^{-\frac12}B(B^{-1}-\tilde P(\tilde P^HB \tilde P)^{-1}\tilde P^H + D)B \widehat X^{-\frac12}\big)
        \\ &= 1 - \lambda_{n_c+1} \big(  \widehat X^{-\frac12}B(B^{-1}-\tilde P(\tilde P^HB \tilde P)^{-1}\tilde P^H)B \widehat X^{-\frac12} 
        + \widehat X^{-\frac12}BDB \widehat X^{-\frac12}\big)
    \end{align*}
    where we used that $\widehat X^{-1}$ is HPD (cf. \cref{thm:B-norm Esouth 1-lambdamin}). Both of the matrices of the sum in the last equation are Hermitian and $\widehat X^{-\frac12}BDB \widehat X^{-\frac12}$ is singular with rank smaller than $\tilde n_c - n_c$. Finally, by \cref{lem:eigenval inequality HornJohnson} it holds
    \begin{align*}
        \norm{E^{\nu, \nu}(&P,R)}_B = 1 - \lambda_{n_c+1} \big(  \widehat X^{-\frac12}B(B^{-1}-\tilde P(\tilde P^HB \tilde P)^{-1}\tilde P^H + D)B \widehat X^{-\frac12}\big)
        \\ &=1 - \lambda_{n_c+1} \big(  \widehat X^{-\frac12}B(B^{-1}-\tilde P(\tilde P^HB \tilde P)^{-1}\tilde P^H)B \widehat X^{-\frac12} + \widehat X^{-\frac12}BDB \widehat X^{-\frac12}\big)
        \\ &\geq 1 - \lambda_{n_c+1 +\rank(D)} \big(  \widehat X^{-\frac12}B(B^{-1}-\tilde P(\tilde P^HB \tilde P)^{-1}\tilde P^H)B \widehat X^{-\frac12}\big)
        \\ &\geq 1 - \lambda_{n_c+1 +(\tilde n_c - n_c)} \big(  \widehat X^{-\frac12}B(B^{-1}-\tilde P(\tilde P^HB \tilde P)^{-1}\tilde P^H)B \widehat X^{-\frac12}\big)
        \\ &= 1 - \lambda_{\tilde n_c+1} \big(  \widehat X^{-1}B(I-\tilde P(\tilde P^HB \tilde P)^{-1}\tilde P^HB)\big)
        \\ &= 1 - \lambda_{\tilde n_c+1} \big(  \widehat X^{-1}B(I-\Pi_B(\tilde P, \tilde P))\big) = \norm{E^{\nu,\nu}(\tilde P, \tilde R)}_B.
    \end{align*}
\end{proof}

\begin{remark}
    Note that we require $\Pi_A$ to be $B$-orthogonal. Without this assumptions, we do not achieve such a result as \cref{exam:counterexample_comp} showed. Since $M^{-1}A$ is $B$-normal and $\Pi_A$ is $B$-orthogonal, we can only consider certain choices of $B$. The possible choices for $B$ were already discussed in \cref{rem:B-normal and B-ortho}.
\end{remark}

The result also holds for the operator $E_+^{1,1}$ with an identical proof. 

\begin{theorem}\label{thm:comp Eplus}
    Let the smoothing assumption be satisfied, $I-\hatMinvB$ be nonsingular and $\tilde n_c \in \N$ with $n_c \leq \tilde n_c < n$. Further, let $P,R \in \Cnnc$ and $\tilde P, \tilde R \in \C^{n \times \tilde n_c}$ be such that $R^HAP$ and $\tilde R^HA \tilde P$ are nonsingular and $\Pi_A(P,R)$, $\Pi_A(\tilde P, \tilde R)$ are $B$-orthogonal. If $\calR(R) \subseteq \calR(\tilde R)$ or $\calR(P) \subseteq \calR(\tilde P)$ is fulfilled, then it holds
    \begin{equation*}
        \norm{E_+^{\nu_1, \nu_2}(\tilde P,\tilde R)}_B \leq \norm{E_+^{\nu_1,\nu_2}(P,R)}_B
    \end{equation*}
    for any combination $(\nu_1, \nu_2) \in \{(1,0),(0,1),(1,1) \}$.
\end{theorem}
\begin{proof}
    The proof is analogous to the proof of \cref{thm:comp Esouth} with the small changes that uses of \cref{thm:B-norm Esouth 1-lambdamin} are replaced by using \cref{thm:B-norm Eplus 1-lambdamin} and $\widehat X^{-1} = \widehat M^{-1}$ since $\nu_1 = \nu_2 = 1$. 
\end{proof}

This again demonstrates that the operator $E_+^{1,1}$ achieves the same results as $E^{\nu,\nu}$ without the additional assumption that $M^{-1}A$ is $B$-normal.

\section{Convergence bounds for V-cycle multigrid} \label{sec:MG}

So far, here and in \cite{NabRoo2026a} we considered the  error propagation matrices $\Eplus$ \cref{eq:Eplus} and $\Esouth$ \cref{eq:Esouth} to study algebraic two-grid methods. Now, we investigate V-cycle AMG methods and extend McCormick's landmark V-cycle bound for HPD matrices \cite{McC1984, McC1985} to nonsymmetric indefinite problems. It is known from \cite{NapNot2010a} that McCormick's bound is the sharpest bound one can obtain for HPD matrices.

We consider $L \in \N$ coarse spaces. Here, $k = 0$ is the finest level, while $k = L$ is the coarsest level. Accordingly, we set $n = n_0 > n_1 > \dots > n_L$ with $n_i \in \N$. Given an arbitrary HPD matrix $B \in \Cnn$ we define $A_0 =A \in \C^{n_0 \times n_0}$ and $B_0 = B \in \C^{n_0 \times n_0}$. In contrast to the previous results, we do not discuss the $B$-orthogonality of the coarse-grid correction as a necessary condition for any convergence result but simply assume that the coarse-grid correction is orthogonal with respect to some inner product at each level of the multilevel algorithm. Thus, we start with a full rank restriction operator $R_0 \in \C^{n_0 \times n_1}$ or interpolation operator $P_0 \in \C^{n_0 \times n_1}$ and employ a series of operators for $k = 0, \dots, L-1$ as follows
\begin{align}
    &\text{choose } R_k, P_k \in \C^{n_k \times n_{k+1}} \text{ such that } \calR(P_k) = \calR(B_k^{-1}A_k^H R_k) \nonumber 
    \\ & A_{k+1} = R_k^HA_k P_k \in \C^{n_{k+1} \times n_{k+1}} \label{eq:multilevel V-cycle}
    \\ & B_{k+1} = P_k^H B_k P_k \in \C^{n_{k+1} \times n_{k+1}}. \nonumber
\end{align}

We know by \cref{lem:equiv_cond Pi B-ortho} that this procedure yields compatible transfer operators $R_k$ and $P_k$ and therefore results in a $B_k$-orthogonal coarse-grid correction $\Pi_{A_k}(P_k,R_k)$. For every level $k = 0, \dots, L$ it holds $P_k = B_k^{-1}A_k^{H}R_kS_k$ for some nonsingular matrix $S_k \in \C^{n_{k+1} \times n_{k+1}}$ and hence
\begin{equation}\label{eq:relation_Bk_Ak}
    A_{k+1} = R_k^HA_kP_k =  S_k^{-H} P_k^HB_kA_k^{-1}A_kP_k = S_k^{-H}P_k^HB_kP_k = S_k^{-H} B_{k+1}.
\end{equation}
The matrix $P_k^HB_kP_k$ is HPD and since $S_k$ is nonsingular, so it $A_{k+1}$. Therefore, for $k = 0, \dots, L-1$ we have
\begin{align*}
    \Pi_{A_k}(P_k, R_k) &= P_kA_{k+1}^{-1}R^H_kA_k = P_k (S_k^{-H}P_k^HB_kP_k)^{-1} S_k^{-H}P_k^HB_kA_k^{-1}A_k 
    \\ &= P_k (P_k^HB_kP_k)^{-1}P_k^H B_k = \Pi_{B_k}(P_k,P_k)
\end{align*}
which is a $B_k$-orthogonal projection onto $\calR(P_k)$ along $\calR(P_k)^{\bot_{B_k}}$. 

The above multilevel procedure basically symmetrizes the coarse-grid matrix $A_k$ in every step. If at some step $k$ we have $P_k = B_k^{-1}A_k^HR_k$, meaning the change of basis matrix $S_k$ satisfies $S_k = I_{n_{k+1}}$, then $A_{k+1} = B_{k+1}$ and we are in the classical HPD setting. For any subsequent level $k$ we can choose $P_k = R_k$ and the remaining multilevel steps are simply the classical HPD V-cycle steps. This is particularly interesting for the first step, that is $k = 0$. If $P_0$ or $R_0$ is given and the other is chosen such that $P_0 = B_0^{-1}A_0^HR_0 = B^{-1}A^HR_0$, then $A_1 = B_1$ is HPD and in the following levels $k$ we can set $P_k = R_k$ and obtain the classical HPD V-cycle. In this case, the techniques in \cref{eq:multilevel V-cycle} are not necessary at all. However, to demonstrate the proof strategies and to illustrate the idea for the nonsymmetric multilevel V-cycle algorithm, we still include these techniques in our analysis.   

We also choose nonsingular smoothers $M_k^{-1} \in \C^{n_k \times n_k}$ for every level $k = 0, \dots, L-1$ such that the smoothing property 
\begin{equation}\label{eq:Vcycle smoothing_assumpt level-wise}
    \norm{I-M^{-1}_kA_k}_{B_k}<1 
\end{equation}
is satisfied at each level $k$. With this we can define the level-wise error propagation matrices for both error operators in a recursive manner starting with $E_{L,+}^{\nu_1, \nu_2} = 0 = E_{L}^{\nu_1, \nu_2} \in \C^{n_L \times n_L}$ and for $k= L-1, \dots ,0$ we set
\begin{align}
    E_{k,+}^{\nu_1, \nu_2} &= (I-(M^{-1}_k A_k)^+)^{\nu_2}(I-P_{k}(I-E_{k+1,+}^{\nu_1, \nu_2})A_{k+1}^{-1}R_k^HA_k)(I-M^{-1}_k A_k)^{\nu_1}\label{eq:Eplus levelwise} \\
    E_{k}^{\nu_1, \nu_2} &= (I-M^{-1}_k A_k)^{\nu_2}(I-P_{k}(I-E_{k+1}^{\nu_1, \nu_2})A_{k+1}^{-1}R_k^HA_k)(I-M^{-1}_k A_k)^{\nu_1} \label{eq:Esouth levelweise}
\end{align}
where the adjoint is taken with respect to $B_k$. We have $E_{k,+}^{\nu_1, \nu_2}, E_{k}^{\nu_1, \nu_2} \in \C^{n_k \times n_k}$ on every level $k$. If a smoothing step $M_k^{-1}A_k$ is $B_k$-orthogonal, the two operators again coincide on level $k$. Notice that we always assume $\Pi_{A_k}$ to be $B_k$-orthogonal. 

For the classic HPD case with $B_k = A_k$ HPD it holds $P_k = R_k$ and the level-wise operator $E_{k,+}^{\nu_1,\nu_2}$ reduces to the well known level-wise error operator of the HPD case. This immediately shows that $E_{k,+}^{\nu_1,\nu_2}$ again is the natural generalization of the error operator in the HPD case. 

Similar to the approach and notation in \cite{McC1984, McC1985}, we decompose the level-wise error propagators $ E_{k,+}^{\nu_1, \nu_2}$ into fine-to-coarse-cycles and coarse-to-fine-cycles. Starting with the coarsest level, $E_{L,+}\backslash^{\nu_1} \in \C^{n_L \times n_L}$ denote the fine-to-coarse-cycle and $E_{L,+}\slash^{\nu_2} \in \C^{n_L \times n_L}$ denote the coarse-to-fine-cycle. We set $E_{L,+}\backslash^{\nu_1} = 0 = E_{L,+}\slash^{\nu_2}$ and for $k = L-1, \dots, 0$ define
\begin{align*}\label{eq:def_Ekslashp_Ekbackp}
    \Ekslashpnu &= (I-(M^{-1}_k A_k)^+)^{\nu_2}(I-P_{k}(I-\Ekoneslashpnu)A_{k+1}^{-1}R_k^HA_k) \in \C^{n_k \times n_k},
    \\ \Ekbackpnu &= (I-P_{k}(I-\Ekonebackpnu)A_{k+1}^{-1} R_k^H A_k)(I-M^{-1}_kA_k)^{\nu_1}\in \C^{n_k \times n_k}.
\end{align*}

Now, we have a fundamental lemma showing properties similar to those of the two-grid case as in \cref{thm:properties Eplus Esouth}.

\begin{lemma}\label{lem:prop_Vcycle_error}
    For $k = L, \ldots, 0$ we have
    \begin{enumerate}[1)]
        \item $ E_{k,+}^{\nu_1, \nu_2} = \Ekslashpnu \cdot \Ekbackpnu$.
        \item For any $\nu \in \N$ it holds $(E_{k,+}^{\nu}/)^+ = E_{k,+}^\nu\backslash$ where the adjoint is with respect to $B_k$.
        \item For $\nu_1 = \nu_2 = \nu$ it holds $\norm{E_{k,+}^{\nu, \nu}}_{B_k} = \norm{E_{k,+}^{\nu}/}_{B_k}^2 = \norm{E_{k,+}^\nu\backslash}_{B_k}^2$. \label{lem:enum:norm_eq Esouth}
    \end{enumerate}
\end{lemma}

\begin{proof}
For simplicity we set $G_k = (I-(M^{-1}_k A_k)^+)$ and $H_k = (I-M_k^{-1}A_k)$. Then it holds $H_k^+ = G_k$. Since $\calR(P_k) = \calR(B_k^{-1}A_k^HR_k)$, for each level $k =0, \dots, L$ there exists a nonsingular matrix $S_k \in \C^{n_{k+1} \times n_{k+1}}$ such that $P_k = B_k^{-1}A_k^HR_kS_k$ and $A_{k+1} = S_k^{-H}B_{k+1}$ holds (see \cref{eq:relation_Bk_Ak}).

We prove \emph{1)} by induction. The base case $E_{L,+}^{\nu_1, \nu_2} = 0 = E_{L,+}^{\nu_2}/ \cdot E_{L,+}^{\nu_1}\backslash$ is clear. Assume that $E_{k+1,+}^{\nu_1, \nu_2} = \Ekoneslashpnu \cdot \Ekonebackpnu$ holds for some $k+1 \in \{L+1, \ldots ,1 \}$. Then we have
\begin{align*}
    \Ekslashpnu &\cdot \Ekbackpnu = G_k^{\nu_2}\left(I-P_{k}(I-\Ekoneslashpnu)A_{k+1}^{-1}R_k^HA_k^{\nu_1}\right) \cdot
    \\ & \hspace*{3.5cm} \left(I-P_{k}(I-\Ekonebackpnu)A_{k+1}^{-1}R_k^HA_k\right)H_k^{\nu_1}
    \\ &= G_k^{\nu_2}\left( (I-\Pi_{A_k})+P_{k}\Ekoneslashpnu A_{k+1}^{-1}R_k^H A_k \right) \cdot 
    \\ & \hspace*{2cm} \left((I-\Pi_{A_k})+P_{k}\Ekonebackpnu A_{k+1}^{-1}R_k^H A_k \right)H_k^{\nu_1}
    \\ &= G_k^{\nu_2}\left( (I-\Pi_{A_k})+ P_k \Ekoneslashpnu A_{k+1}^{-1}R_k^H A_k P_k \Ekonebackpnu A_{k+1}^{-1}R_k^HA_k\right)H_k^{\nu_1}
    \\ &= G_k^{\nu_2}\left( (I-\Pi_{A_k}) + P_k \Ekoneslashpnu\Ekonebackpnu A_{k+1}^{-1}R_k^HA_k\right)H_k^{\nu_1}
    \\ &= G_k^{\nu_2}\left( (I-\Pi_{A_k}) + P_k E_{k+1,+}^{\nu_1, \nu_2} A_{k+1}^{-1}R_k^HA_k\right)H_k^{\nu_1}
    \\ &= G_k^{\nu_2}(I- P_k(I-E_{k+1,+}^{\nu_1, \nu_2})A_{k+1}^{-1}R_k^HA_k)H_k^{\nu_1} = E_{k,+}^{\nu_1, \nu_2}
\end{align*}
where we used $(I-\Pi_{A_k})P_k = 0$. 

For \emph{2)} let $\nu \in \N$. We use induction again. The base case is clear since $(E_{L,+}^{\nu}/)^+ = 0 = E_{L,+}^{\nu}\backslash$. Assume that 
\begin{equation*}
    (E_{k+1,+}^\nu/)^+ = B_k^{-1} (E_{k+1,+}^\nu/)^H B_k = E_{k+1,+}^\nu\backslash
\end{equation*}
holds for some $k+1 \in \{L+1, \dots, 1\}$. Then this is equivalent to $B_k^{-1}(E_{k+1,+}^\nu/)^H = E_{k+1,+}^\nu\backslash B_k^{-1}$. It holds
\begin{align*}
    (E_{k,+}^\nu/)^+ &= \left(I-P_{k}(I-E_{k+1,+}^\nu/)A_{k+1}^{-1}R_k^HA_k \right)^+ \left(G_k^{\nu}\right)^+
    \\ &= B_k^{-1} (I-A_k^HR_kA_{k+1}^{-H}(I-E_{k+1,+}^\nu/)^HP_k^H)B_k H_k^\nu
    \\ &= (I-B_k^{-1}A_k^HR_kA_{k+1}^{-H}(I-E_{k+1,+}^\nu/)^HP_k^HB_k) H_k\nu
    \\ &= (I-P_kS_k^{-1} (S_k^{-H} B_{k+1})^{-H}(I-E_{k+1,+}^\nu/)^H S_k^HR_k^HA_k) H_k^\nu
    \\ &= (I-P_kS_k^{-1}S_k B_{k+1}^{-1}(I-E_{k+1,+}^\nu/)^H S_k^HR_k^HA_k) H_k^\nu
    \\ &= (I-P_k (I-E_{k+1,+}^\nu\backslash) B_{k+1}^{-1} S_k^HR_k^HA_k) H_k^\nu
    \\ &= (I-P_k (I-E_{k+1,+}^\nu\backslash) A_{k+1}^{-1} R_k^HA_k) H_k^\nu = E_{k,+}^\nu\backslash.
\end{align*}

For \emph{3)} let $\nu_1 = \nu_2 = \nu$. With \emph{1)} and \emph{2)} we get
\begin{equation*}
    \norm{E_{k,+}^{\nu, \nu}}_{B_k} = \norm{E_{k,+}^\nu/  E_{k,+}^\nu\backslash}_{B_k} = \norm{E_{k,+}^\nu/ (E_{k,+}^\nu/)^+}_{B_k} = \norm{E_{k,+}^\nu/}_{B_k}^2
\end{equation*}
and
\begin{equation*}
    \norm{E_{k,+}^{\nu, \nu}}_{B_k} = \norm{E_{k,+}^\nu/ E_{k,+}^\nu\backslash}_{B_k} = \norm{(E_{k,+}^\nu\backslash)^+ E_{k,+}^\nu\backslash}_{B_k} = \norm{E_{k,+}^\nu\backslash}_{B_k}^2
\end{equation*}
for all $k = L, \dots, 0$.
\end{proof}

Next, we extend McCormick's V-cycle bound \cite{McC1984,McC1985} for HPD matrices to nonsymmetric indefinite matrices. Since the following results require $\nu_1 = \nu_2 = 1$, we omit the superindices and simply write $E_{k,+}$, $\Ekslashp$ and $\Ekbackp$ for $E_{k,+}^{1,1}$, $E_{k,+}^1/$ and $E_{k,+}^1\backslash$, respectively. 

\begin{theorem} \label{thm:V-cycle Eplus}
    Consider the V-cycle algorithm as in \cref{eq:multilevel V-cycle}. For every level $k = 0, \dots, L-1$ let the smoothing property 
    \begin{equation}\label{eq:Vcycle smoothing_cond McC}
        \norm{(I-(M^{-1}_k A_k)^+)e}_{B_k}^2 \leq \alpha \norm{(I-\Pi_{A_k})e}_{B_k}^2 + \norm{\Pi_{A_k} e}_{B_k}^2
    \end{equation}
    be satisfied for some $\alpha \in [0,1)$ and all $e\in \mathbb{C}^{n_k}$. Then it holds
    \begin{equation*}
        \norm{E_{k,+}}_{B_k} \leq \alpha
    \end{equation*}
    for all levels $k = 0, \dots, L$.
\end{theorem}
\begin{proof}
    The proof is done by induction from $L$ to $0$. Since $E_{L,+} =0$, the base case is trivial. Let
    \begin{equation*}
        \norm{E_{k+1,+}}_{B_{k+1}}  \leq \alpha 
    \end{equation*}
    be satisfied for some $k+1 \in \{L+1, \dots, 1\}$. Let $e \in \C^{n_k}$. Using the smoothing property~\cref{eq:Vcycle smoothing_cond McC} we get
    \begin{align*}
        \norm{\Ekslashp e}_{B_k}^2 &= \norm{(I-(M^{-1}_kA_k)^+)(I-P_k(I-\Ekoneslashp)A_{k+1}^{-1}R_k^HA_k)e}_{B_k}^2
        \\ &= \norm{(I-(M^{-1}_kA_k)^+)((I-\Pi_{A_k})e + P_k\Ekoneslashp A_{k+1}^{-1}R_k^HA_ke)}_{B_k}^2
        \\ &\leq \alpha \norm{(I-\Pi_{A_k})^2e}_{B_k}^2 + \norm{\Pi_{A_k} P_k\Ekoneslashp A_{k+1}^{-1}R_k^HA_ke}_{B_k}^2
        \\ &= \alpha \norm{(I-\Pi_{A_k})e}_{B_k}^2 + \norm{P_k\Ekoneslashp A_{k+1}^{-1}R_k^HA_ke}_{B_k}^2.
    \end{align*}
    Since $B_{k+1} = P_k^HB_kP_k$, the second term can be estimated further by
    \begin{align*}
        \norm{P_k\Ekoneslashp A_{k+1}^{-1}R_k^HA_ke}_{B_k}^2 &= \norm{\Ekoneslashp A_{k+1}^{-1}R_k^HA_ke}_{B_{k+1}}^2
        \\ &\leq \norm{\Ekoneslashp}_{B_{k+1}}^2 \norm{A_{k+1}^{-1}R_k^HA_ke}_{B_{k+1}}^2
        \\ &= \norm{\Ekoneslashp}_{B_{k+1}}^2 \norm{P_k A_{k+1}^{-1}R_k^HA_ke}_{B_{k}}^2
        \\ &= \norm{\Ekoneslashp}_{B_{k+1}}^2 \norm{\Pi_{A_k} e}_{B_{k}}^2.
    \end{align*}
    Since $\Pi_{A_k}e$ and $(I-\Pi_{A_k})e$ are orthogonal in the $B_k$-inner product, we get
    \begin{equation*}
        \norm{(I-\Pi_{A_k})e}_{B_k}^2 + \norm{\Pi_{A_k}e}_{B_k}^2 = \norm{(I-\Pi_{A_k} + \Pi_{A_k})e}_{B_k}^2 = \norm{e}_{B_k}^2.
    \end{equation*}
    Finally, with the induction hypothesis and \cref{lem:prop_Vcycle_error}~(\labelcref{lem:enum:norm_eq Esouth}) we obtain
    \begin{align*}
        \norm{\Ekslashp e}_{B_k}^2 &\leq \alpha \norm{(I-\Pi_{A_k})e}_{B_k}^2 + \norm{\Ekoneslashp}_{B_{k+1}}^2 \norm{\Pi_{A_k} e}_{B_{k}}^2
        \\ &\leq \max\{\alpha, \norm{\Ekoneslashp}_{B_{k+1}}^2\}(\norm{(I-\Pi_{A_k}) e}_{B_k}^2 + \norm{\Pi_{A_k} e}_{B_k}^2)
        \\ &= \max\{\alpha, \norm{E_{k+1,+}}_{B_{k+1}}\}\norm{e}_{B_k}^2 \leq \alpha \norm{e}_{B_k}^2
    \end{align*}
     Thus, it holds
    \begin{displaymath}
        \norm{E_{k,+}}_{B_k} = \norm{\Ekslashp}_{B_k}^2 = \sup_{e \neq 0} \frac{\norm{\Ekslashp e}_{B_k}^2}{\norm{e}_{B_k}^2} \leq \alpha.
    \end{displaymath}
\end{proof}

Note that the smoothing property~\cref{eq:Vcycle smoothing_cond McC} implies the usual smoothing assumption~\cref{eq:Vcycle smoothing_assumpt level-wise} showing that \cref{eq:Vcycle smoothing_cond McC} is a stronger assumption.

In the above result $\alpha$ is independent of the level $k$. We can also choose the smallest possible parameter depending on each level such that the smoothing property~\cref{eq:Vcycle smoothing_cond McC} is fulfilled. This yields the following result which shows that the convergence constant $\alpha$ is the maximum over all level wise constants.

\begin{corollary} \label{coro:V-cycle Eplus levelwise constant}
    Consider the V-cycle algorithm as in \cref{eq:multilevel V-cycle}. Let $\nu_1 = \nu_2 = 1$ and for every level $k = L, \dots, 0$ let 
    \begin{equation}\label{eq:alpha_k}
        \alpha_k = \sup_{(I-\Pi_{A_k})e\neq 0} \frac{\norm{(I-(M_k^{-1}A_k)^+)e}_{B_k}^2-\norm{\Pi_{A_k} e}_{B_k}^2}{\norm{(I-\Pi_{A_k})e}_{B_k}^2}.
    \end{equation}
    Assume that $\alpha_k  \in [0,1)$ and the usual level wise smoothing assumption \cref{eq:Vcycle smoothing_assumpt level-wise} is satisfied. Then the smoothing property~\cref{eq:Vcycle smoothing_cond McC} is fulfilled with $\alpha = \max_{k=0, \dots, L} \alpha_k$ and it holds 
    \begin{equation*}
        \norm{E_{0,+}}_B \leq \alpha =  1 - \frac{1}{C_V}
    \end{equation*}
    with 
    \begin{equation*}
        C_V = \max_{k=0, \dots, L} \sup_{e \neq 0} \frac{\norm{(I-\Pi_{A_k})B_k^{-1}e}_{B_k}^2}{\norm{e}_{\widehat M_k^{-1}}^2}.
    \end{equation*}
\end{corollary}

\begin{proof}
    Let $\alpha_k$ be defined as in \cref{eq:alpha_k}. Then set $\alpha = \max_{k=0, \dots, L} \alpha_k \in [0,1)$. Multiplying \cref{eq:alpha_k} by $\norm{(I-\Pi_{A_k})e}_{B_k}^2$ for any $e \in \C^{n_k}$ with $(I-\Pi_{A_k})e \neq 0$ yields 
    \begin{align*}
        \norm{(I-(M_k^{-1}A_k)^+)e}_{B_k}^2 &\leq \alpha_k \norm{(I-\Pi_{A_k})e}_{B_k}^2 + \norm{\Pi_{A_k}e}_{B_k}^2 
        \\ & \leq \alpha \norm{(I-\Pi_{A_k})e}_{B_k}^2 + \norm{\Pi_{A_k}e}_{B_k}^2.
    \end{align*}
    For all $e \in \C^{n_k}$ with $(I-\Pi_{A_k})e = 0$ it holds $e = \Pi_{A_k}e$ and with the smoothing assumption \cref{eq:Vcycle smoothing_assumpt level-wise} it follows
    \begin{align*}
        \norm{(I-(M_k^{-1}A_k)^+)e}_{B_k}^2 &\leq \norm{I-(M_k^{-1}A_k)^+)}_{B_k}^2 \norm{\Pi_{A_k}e}_{B_k}^2 
        \\ &< \norm{\Pi_{A_k}e}_{B_k}^2 = \alpha \norm{(I-\Pi_{A_k})e}_{B_k}^2 + \norm{\Pi_{A_k}e}_{B_k}^2.
    \end{align*}
    Therefore, the smoothing property~\cref{eq:Vcycle smoothing_cond McC} is satisfied with the constant $\alpha$ and for all $e \in \C^{n_k}$.
    Hence we get
    \begin{equation*}
        \norm{E_{0,+}}_B \leq \alpha.
    \end{equation*}
    Due the $B_k$-orthogonality of $\Pi_{A_k}$ it holds
    \begin{align*}
        \alpha_k &= \sup_{(I-\Pi_{A_k})e\neq 0} \frac{\norm{(I-(M_k^{-1}A_k)^+)e}_{B_k}^2-\norm{\Pi_{A_k} e}_{B_k}^2}{\norm{(I-\Pi_{A_k})e}_{B_k}^2} 
        \\ &= \sup_{(I-\Pi_{A_k})e\neq 0} \frac{\norm{(I-(M_k^{-1}A_k)^+)e}_{B_k}^2-\norm{e}_{B_k}^2+\norm{(I-\Pi_{A_k})e}_{B_k}^2}{\norm{(I-\Pi_{A_k})e}_{B_k}^2}
        \\ &=  1- \inf_{(I-\Pi_{A_k})e\neq 0} \frac{\norm{e}_{B_k}^2 - \norm{(I-(M_k^{-1}A_k)^+)e}_{B_k}^2}{\norm{(I-\Pi_{A_k})e}_{B_k}^2}. 
    \end{align*}
    The nominator of the fraction in the last line can be further computed as follows
    \begin{align*}
        \norm{e}_{B_k}^2- &\phantom{\, } \norm{(I-(M_k^{-1}A_k)^+)e}_{B_k}^2 
        \\ &= e^H B_k e - e^H(I-(M_k^{-1}A_k)^+)^HB_k(I-(M_k^{-1}A_k)^+)e
        \\ &= e^H \left( B_k - (I-B_k(M^{-1}_kA_k)B_k^{-1}) B_k(I-(M_k^{-1}A_k)^+)\right)e
        \\ &= e^H \left(B_k - B_k(I-M_k^{-1}A_k)(I-(M_k^{-1}A_k)^+) \right)e
        \\ &= e^HB_k (I-(I-M_k^{-1}A_k)(I-(M_k^{-1}A_k)^+))e
        \\ &= e^H B_k (I-(I- \widehat M_k^{-1}B_k))e
        \\ &= e^H B_k \widehat M_k^{-1}B_ke  = \norm{B_ke}_{\widehat M_k^{-1}}^2
    \end{align*}
    where we define $\widehat M_k^{-1}B_k = M_k^{-1}A_k + (M_k^{-1}A_k)^+ - M_k^{-1}A_k (M_k^{-1}A_k)^+$. This simply is a level-wise definition of $\hatMinvB$ in \cref{eq:hatMinvB}. Since the smoothing assumption \cref{eq:Vcycle smoothing_assumpt level-wise} is satisfied at every level $k$, the matrix $\widehat M_k^{-1}$ is HPD (cf. \cref{thm:smoothing_assumpt_BatNab}). Then it holds
    \begin{align*}
        \alpha_k &= 1 - \inf_{(I-\Pi_{A_k})e\neq 0} \frac{\norm{B_ke}_{\widehat M_k^{-1}}^2}{\norm{(I-\Pi_{A_k})e}_{B_k}^2}
        = 1- \left( \sup_{e\neq 0} \frac{\norm{(I-\Pi_{A_k})e}_{B_k}^2}{\norm{B_ke}_{\widehat M_k^{-1}}^2} \right)^{-1}
        \\ &= 1- \left( \sup_{r\neq 0} \frac{\norm{(I-\Pi_{A_k})B_k^{-1}r}_{B_k}^2}{\norm{r}_{\widehat M_k^{-1}}^2} \right)^{-1}.
    \end{align*}
    The second equality holds because the infimum is always positive. Finally, we get
     \begin{align*}
        \alpha &= \max_{k=0, \dots, L} \alpha_k = \max_{k=0, \dots, L} 1- \left( \sup_{e\neq 0} \frac{\norm{(I-\Pi_{A_k})B_k^{-1}e}_{B_k}^2}{\norm{e}_{\widehat M_k^{-1}}^2} \right)^{-1}
        \\ &= 1-\min_{k=0, \dots, L} \left( \sup_{e\neq 0} \frac{\norm{(I-\Pi_{A_k})B_k^{-1}e}_{B_k}^2}{\norm{e}_{\widehat M_k^{-1}}^2} \right)^{-1}
        \\ &= 1 - \left(\max_{k=0, \dots, L} \sup_{e\neq 0} \frac{\norm{(I-\Pi_{A_k})B_k^{-1}e}_{B_k}^2}{\norm{e}_{\widehat M_k^{-1}}^2} \right)^{-1} = 1- C_V^{-1}
    \end{align*}
    and $\norm{E_{0,+}}_B \leq \alpha = 1 - C_V^{-1} $.
\end{proof}

\begin{remark}
    Assuming that the matrix $I_{n_k}-\widehat M_k^{-1}B_k$ is nonsingular at every level $k = L, \dots, 0$, we can relate the constant $C_V$ to the constant $K_{G,C}(P)$ from \cref{eq:def KG} for certain HPD matrices $G$ and $C$. For the interpolation $P_k$ from the V-cycle algorithm \cref{eq:multilevel V-cycle} we define $Q_{B_k} = P_k(P_k^HB_kP)^{-1}P_k$, analogous to $Q_B$ in the proof of \cref{thm:B-norm Esouth 1-lambdamin}. Then it holds $\Pi_{A_k} = Q_{B_k}B_k$ and we obtain
    \begin{align*}
        \norm{(I-\Pi_{A_k})B_k^{-1}e}_{B_k}^2 &= \norm{(I-Q_{B_k}B_k)B_k^{-1}e}_{B_k}^2 
        \\ &= e^HB_k^{-1}(I-Q_{B_k}B_k)^HB_k(I-Q_{B_k}B_k)B_k^{-1}e 
        \\ &= e^H(I-Q_{B_k}B_k)B_k^{-1}e.
    \end{align*}
    Together with the above proof and the Courant-Fisher theorem this implies
    \begin{align*}
        1- \alpha_k &= \sup_{e\neq 0} \frac{\norm{(I-\Pi_{A_k})B_k^{-1}e}_{B_k}^2}{\norm{e}_{\widehat M_k^{-1}}^2} = \sup_{e\neq 0} \frac{e^H(I-Q_{B_k}B_k)B_k^{-1}e}{e^H\widehat M_k^{-1}e}
        \\ &= \sup_{v \neq 0} \frac{v^H\widehat M_k^{\frac12}(I-Q_{B_k}B_k)B_k^{-1}\widehat M_k^\frac12 v}{v^Hv} = \lambda_{max} \left( \widehat M_k^\frac12(I-Q_{B_k}B_k)B_k^{-1}\widehat M_k^\frac12 \right)
        \\ &= \lambda_{max} \left( B_k^{-1}\widehat M_k(I- \Pi_{A_k}) \right) = K_{\widehat M_k, B_k}(P_k).
    \end{align*}
    Here we used the same reasoning as in the proof of \cref{thm:B-norm Esouth 1-K}. Finally, we can conclude 
    \begin{equation*}
        C_V = (1-\alpha)^{-1} = \frac{1}{1-\max_{k} \alpha_k} = \frac{1}{\min_k K_{\widehat M_k, B_k}(P_k)} = \max_k \frac{1}{K_{\widehat M_k, B_k}(P_k)}.
    \end{equation*}
\end{remark}

\Cref{thm:B-norm Eplus 1-K} only holds for the operator $E_{k,+}^{1,1}$. However, we can obtain the same results for $E_{k}^{1,1}$ if we assume that $M_k^{-1}A_k$ is $B_k$-orthogonal at every level.  

\begin{corollary}
    Consider the V-cycle algorithm as in \cref{eq:multilevel V-cycle}. Let $M_k^{-1}A_k$ be $B_k$-orthogonal for every level $k = 0, \dots, L$ and and let the smoothing property
    \begin{equation*}
        \norm{(I-M^{-1}_k A_k)e}_{B_k}^2 = \norm{(I-(M^{-1}_k A_k)^+)e}_{B_k}^2 \leq \alpha \norm{(I-\Pi_{A_k})e}_{B_k}^2 + \norm{\Pi_{A_k} e}_{B_k}^2
    \end{equation*}
    be satisfied for some $\alpha \in [0,1)$ and all $e\in \mathbb{C}^{n_k}$. Then it holds
    \begin{equation*}
        \norm{E_k}_{B_k} \leq \alpha
    \end{equation*}
    for every level $k = 0, \dots, L$.
\end{corollary}

\begin{remark}
    If in the V-cycle algorithm \cref{eq:multilevel V-cycle} the interpolation $P_{k-1}$ is chosen as $P_{k-1} = B_{k-1}^{-1}A_{k-1}^HR_{k-1}$, then $M_k^{-1}A_k$ is $B_k$-orthogonal if and only if $M_k$ is Hermitian. To see this we first notice that the change of basis matrix $S_{k-1}$ that transforms $\calR(B_{k-1}^{-1}A_{k-1}^HR_{k-1})$ into $\calR(P_{k-1})$ is the identity, that is $S_{k-1} = I_{n_k}$. Hence, with \cref{eq:relation_Bk_Ak} we obtain
    \begin{equation*}
        (M_k^{-1}A_k)^+ = B_k^{-1}A_k^HM_k^{-H}B_k = B_k^{-1}B_k S_{k-1}^HM_k^{-H}S_{k-1}^HA_k = M_k^{-H}A_k
    \end{equation*}
    which is equal to $M_k^{-1}A_k$ if and only if $M_k$ is Hermitian. This is due to the fact that for $S_{k-1}= I_{n_k}$ it holds $B_k = A_k$ and we are in the HPD case where the $A_k$-adjoint of $M_k^{-1}A_k$ is given by $M_k^{-H}A_k$. 
\end{remark}

\section{Conclusion}

In this paper, we continued our unified framework for nonsymmetric AMG methods started in \cite{NabRoo2026a} which allows the use of arbitrary inner products given by an HPD matrix $B$. We proved more results on the convergence speed of nonsymmetric algebraic two-grid methods for the error operator $\Esouth$ and $\Eplus$, as well as comparison results showing that an increase of the number of coarse variables decreases the $B$-norm of error operators of these methods and thus increasing the convergence speed. We also considered nonsymmetric multilevel methods, in particular the V-cycle, and extended McCormick's landmark V-cycle bound \cite{McC1984,McC1985} to the nonsymmetric and indefinite case.

\bibliographystyle{siamplain}
\bibliography{literature_paper}

\end{document}